\documentclass[11pt]{amsart}
\usepackage{color, amsmath,amssymb, amsfonts, amstext,amsthm, latexsym,mathtools}
\usepackage[active]{srcltx}
\usepackage[titletoc,title]{appendix}
\allowdisplaybreaks

\usepackage{cite}
\usepackage{ bbold }
\usepackage{graphicx}
\usepackage[show]{ed} 
\usepackage[T1]{fontenc}
\usepackage[english]{babel}
\usepackage{enumerate}
\usepackage{babel}
\usepackage[pdftex,breaklinks,colorlinks, citecolor=blue,urlcolor=blue]{hyperref}
\usepackage{verbatim}
\numberwithin{equation}{section}
\usepackage{enumitem}

\usepackage[a4paper, left=1in, right=1in, top=1.5in, bottom=1.5in]{geometry}

\usepackage{relsize}

\newcommand{\im}{\mathrm{Im}}

\newcommand{\re}{\mathrm{Re}}

\newcommand{\R}{{\mathbb R}}

\newcommand{\HH}{ \mathbb{H} }

\newcommand{\ch}{\operatorname{ch}}
\newcommand{\sh}{\operatorname{sh}}

\def \calR   {\mathcal{R}}

\newcommand{\tha}{\operatorname{th}}

\numberwithin{equation}{section}
\newtheorem{theorem}{Theorem}[section]

\newtheorem{lemma}[theorem]{Lemma}
\newtheorem{remark}[theorem]{Remark}
\newtheorem{prop}[theorem]{Proposition}
\newtheorem{coro}[theorem]{Corollary}
\newtheorem{claim}[theorem]{Claim}

\begin{document}
\title[Stability of vortices in Manton's model on $\mathbb{H}^2$]{Equivariant stability of vortices in Manton's Chern-Simons-Schr\"odinger system on the hyperbolic plane}

 \begin{abstract}
In this work we study magnetic vortices on the hyperbolic plane for a Chern-Simons-Schr\"odinger system introduced by Manton. The model can be thought of as the Schr\"odinger analogue of the Abalian-Higgs model. It consists of a system of partial differential equations, where the complex Higgs field $\Phi$ evolves according to a nonlinear Schr\"odinger equation coupled to an electromagnetic field $A$. We restrict attention to the self-dual (Bogomolny) case under equivariance symmetry. For each $m\geq 1$ we prove the asymptotic stability of the equivariant vortex of degree $m$. The main novelties are unraveling the favorable structure of the equations after a nonlinear Darboux transform, and the analysis of the elliptic operator relating the original and the transformed variables.

 \end{abstract}

 \author[O. Landoulsi]{Oussama Landoulsi }
\address{Department of Mathematics \\ University of Massachusetts \\ Amherst, MA 01003, USA}
\email{olandoulsi@umass.edu}
 
 \author[S. Shahshahani]{Sohrab Shahshahani}
\address{Department of Mathematics \\ University of Massachusetts \\ Amherst, MA 01003, USA}
\email{sshahshahani@umass.edu}

\thanks{The second author was partially supported by the Simons Foundation grant 639284. }

 \maketitle 
 \tableofcontents

\section{Introduction}
In this paper we study a gauged Schr\"odinger equation, sometimes known as Manton's model, on $\mathbb{R}\times\mathbb{H}^2$, where $\mathbb{H}^2$ is the hyperbolic plane. Manton's model, which may be viewed as the Schr\"odinger analogue of the Abelian Higgs model, was introduced by Manton in \cite{manton1997} on $\mathbb{R}\times \mathbb{R}^2$. The equations are derived from a Lagrangian which is a Galilean invariant version the Ginzburg-Landau model with kinetic terms that are linear in the first time derivative of the field. See \cite[equation (2.1)]{manton1997}. The corresponding model has also been studied on non-flat Riemann surfaces. See for instance \cite{demouliniStuart09,mantonSutcliffe04,demoulini2007} and references therein.  The time independent version of the equation is the well-known Ginzburg-Landau system. The solutions of the latter, known as vortices, are time-independent, or solitonic, solutions of Manton's model. Here we are interested in the question of stability of Ginzburg-Landau vortices under Manton's Schr\"odinger flow. Specifically, we prove that for each $k\geq1$, the unique equivariant vortex of degree $k$ is asymptotically stable under equivariant perturbations. See Theorem~\ref{main-theo} for the precise version of this statement. \\

The unknowns in Manton's model are a complex scalar field $\Phi$ and a one form $A$, both defined on $\mathbb{R}\times \mathbb{H}^2$. Invariantly, $\Phi$ can be viewed as a section of a trivial line bundle and $A$ as a connection form on the bundle. To describe the equations we will use $(x^1,x^2)$ to denote an arbitrary coordinate system on $\mathbb{H}^2$, and $(t,x^1,x^2)$ to denote the corresponding coordinate system on $\mathbb{R}\times \mathbb{H}^2$. Roman indices will vary over $\{1,2\}$ and Greek indices will vary over 
$\{0,1,2\}$ with $x^0=t$. We use $g$ for the hyperbolic metric on $\mathbb{H}^2$ and use it to raise and lower spatial indices. Associated with $A$ is the curvature two form $F=dA$ and the covariant derivative $D=\nabla-iA$. Here $\nabla$ denotes the Levi-Civita connection on $\mathbb{R}\times \mathbb{H}^2$. The curvature form is decomposed into a magnetic part $B$ and an electric part $E=E_i\,d x^0\wedge d x^i$ as
\begin{equation*}
    F=E_i\,d x^0\wedge d x^i+ B \, dvol,
\end{equation*}
where $dvol$ denotes the volume form on $\mathbb{H}^2$. Here, as in everywhere else in this paper, we have used the convention of summing over repeated indices. With $J$ denoting the complex structure on $\mathbb{H}^2$, Manton's model is given by the equations
\begin{equation}\label{eq:Mantonintro1}
    \begin{cases}
        &iD_t\Phi+\frac{1}{2}D^jD_j\Phi=-\frac{\lambda}{4}(1-|\Phi|^2)\Phi\\
        &2E_j+\partial_jB=-J_j^k\mathrm{Im}(\Phi \overline{D_j\Phi})\\
        &B=\frac{1}{2}(1-|\Phi|^2)
    \end{cases}.
\end{equation}
Here $\lambda$ is a fixed coupling constant. The case $\lambda=1$ is known as the self-dual case and has received considerable attention in both the physics and mathematics literature. See for instance \cite{mantonSutcliffe04, manton1997, jaffeeTaubes80}. Physically, it corresponds to the scenario where there are no forces between static vortices. Mathematically, the static equations in the case $\lambda=1$ have a Bogomolny structure that leads to a classification of solutions. We will henceforth restrict attention to the self-dual case $\lambda=1$ and discuss the structure of the resulting first order Bogomolny equations below.  

 The study of Ginzburg-Landau vortices and the associated time-dependent flows, including Manton's model, on non-flat geometries has a long history. See for instance \cite{mantonSutcliffe04, demoulini2007, demouliniStuart09, stuart94, stuart99, bradlow90, garcia94,  strachan92, samols92, witten1983} and references therein. The case of the hyperbolic geometry is particularly interesting for two reasons. On the one hand there has been growing interest in the study of dispersive equations on hyperbolic spaces in recent years. We refer the reader for instance to the works \cite{AnkerJeanPierfelice09, banica07,banicaCarlesStaffilani08,BaDu07, IS2013, banicaCarlesDuyckaerts08,banicaDuyckaerts14, IPS, LOS1, BorMar1, MT11, MTay12, LLOS2, MWYZ1, wilsonYu25, ZeLi1, ZeLi2} and their references. From this point of view, our work belongs in the realm of soliton stability problems on curved domains. On the other hand, the static Ginzburg-Landau equations on the hyperbolic plane of curvature $-\frac{1}{2}$ arise naturally as a reduction of Yang-Mills instantons under $SO(3)$ cylindrical symmetry. This observation was used by Witten in \cite{witten77} to find the first multi-instanton solutions. Remarkably, in the case of curvature $-\frac{1}{2}$ and self-dual coupling, solutions of the Boglomony equations can be solved to obtain explicit formulas for the vortices. See \cite{mantonSutcliffe04}. For this reason the choice of hyperbolic plane with curvature $-\frac{1}{2}$ is very natural for the study of \eqref{eq:Mantonintro1}. Nevertheless, for notational convenience, we have chosen to work with the more common choice of curvature $-1$ for $\mathbb{H}^2$. It should be noted, however, that our analysis goes through with only minimal notational changes for curvature $-\frac{1}{2}$. In fact, as we already mentioned, for the latter the soliton is explicit, making some computations slightly easier, while for our choice of curvature we are not aware of an explicit formula for the soliton.

 At the surface Manton's model \eqref{eq:Mantonintro1} is similar to the classical Schr\"odinger-Chern-Simons (CSS) equation. The difference is that in the (CSS) the nonlinearity on the right-hand side of the first equation in \eqref{eq:Mantonintro1} is $|\Phi|^2\Phi$. As we will discuss below, in Manton's model finite energy solutions satisfy $|\Phi|\to1$ at spacial infinity, while finite energy solutions of (CSS) decay. Consequently, solutions of Manton's model and (CSS) exhibit different behaviors. For instance, in the case of (CSS) on Euclidean space blowing up solutions have been constructed in \cite{kimKwonwOh20,kim2Kwon23,KimKwon23,kimKwonOh24}. Nevertheless, in view of the affinities of the equations, and the recent attention (CSS) has received, we briefly mention several works on (CSS) and refer the reader to them for further discussion of the related history. Local existence has been studied in \cite{BdBS1,Huh13,LZM18,liuSmithTataru14,liuSmith16}. Results on the long-time behavior without symmetry assumptions include \cite{BdBS1,BdBS2,PusateriOh15}. Under equivariance symmetry much more is known, including finite-time blow up constructions \cite{kimKwonOh24,kim2Kwon23,kimKwonwOh20,KimKwon23}, a subthreshold scattering theorem \cite{liuSmith16}, classification of threshold solutions \cite{LiLiu22}, and soliton resolution \cite{kimKwonOh22}. See also \cite{czubakMillerRoudenko24,ChaeChoe02} for results on some related equations. Finally we mention that Manton's model is a special case of the Zhang-Hansson-Kivelson or ZHK system \cite{zhangHanssonKivelson}. See \cite{mantonSutcliffe04,Hovrathy-Zhang09, demouliniStuart09,zhangHanssonKivelson} for further discussion on the the relation between these models.
\subsection{Structure of the equations in the self-dual case}
We now return to \eqref{eq:Mantonintro1} and describe in more details the special features of the equation in the self-dual case $\lambda=1$. To make matters more explicit we introduce polar coordinates $(r,\theta)$ on $\mathbb{H}^2$ in which the metric $g$ and the volume form $dvol$ take the forms
\begin{equation}
    g= dr\otimes dr +\sh^2r \,d\theta\otimes d\theta,\qquad dvol = \sh r \,dr \wedge d\theta.
\end{equation}
The magnetic field $B$ and the electric field $E$ are given by
\begin{equation}
    B=\frac{1}{\sh r}(\partial_rA_\theta-\partial_\theta A_r),\qquad E_r=\partial_0A_r-\partial_r A_0,\qquad E_\theta=\partial_0A_\theta-\partial_\theta A_0.
\end{equation}
The Laplacian on $\mathbb{H}^2$ is
\begin{equation}
    \Delta= \partial_r^2+\coth r \partial_r+\frac{1}{\sh^2r}\partial_\theta^2.
\end{equation}
Note that for any smooth function $\chi$, equation \eqref{eq:Mantonintro1} is invariant under the gauge transformation
\begin{equation}
    \Phi\mapsto \Phi e^{i\chi}, \qquad A\mapsto A+d\chi.
\end{equation}
As is standard with related models, to remove this ambiguity we will later fix a gauge. Starting with the static equations 
\begin{equation}
    D^jD_j\Phi=-\frac{1}{4}(1-|\Phi|^2)\Phi,\qquad B= \frac{1}{2}(1-|\Phi|^2),
\end{equation}
note that they are the Euler-Lagrange equations of the Ginzburg-Landau action
\begin{equation}
    V(\Phi,A)=\frac{1}{2}\int_{\mathbb{H}^2}\big(B^2+|D\Phi|^2+\frac{1}{4}(1-|\Phi|^2)^2\big)dvol.
\end{equation}
Solutions for which $V(\Phi,A)$ is finite are called finite-energy vortices. The finiteness of energy implies that $|\Phi|$ must tend to one as $r\to\infty$. Note that
\begin{equation}
    \frac{1}{2}\int_{\mathbb{H}^2}B dvol = \frac{1}{2}\lim_{r\to\infty}\int_0^{2\pi}A_\theta(r,\theta)d\theta.
\end{equation}
It is not difficult to see that the latter is always an integer multiple of $\pi$, and is independent of the choice of gauge. That is
\begin{equation}
    \frac{1}{2}\int_{\mathbb{H}^2}B dvol=\pi m,
\end{equation}
for an integer $m$ which we call the degree of $\Phi$. See \cite[Section~3.7]{mantonSutcliffe04} for more details. Using this observation, the energy $V(\Phi,A)$ can be factorized as (this is where the choice $\lambda=1$ comes in; in general $V$ would have $\frac{\lambda}{4}(1-|\Phi|)^2$)
\begin{equation}
    V(\Phi,A)=\frac{1}{2}\int_{\mathbb{H}^2}\Big(4|\overline{\partial}_A\Phi|^2+\big(B-(\frac{1}{2}(1-|\Phi|^2))\big)^2\Big)dvol+m\pi,
\end{equation}
where $\overline{\partial}_A:=\frac{1}{2}(D_r+\frac{i}{\sh r}D_\theta)$ is the covariant Cauchy-Riemann operator. This decomposition leads to the following Bogomolny equations for the solution 
\begin{equation}\label{eq:Bogomolny1}
    \overline{\partial}_A\Phi=0,\qquad B=\frac{1}{2}(1-|\Phi|^2).
\end{equation}
It is known that all finite-energy static solutions can be obtained in this manner. See \cite{jaffeeTaubes80,mantonSutcliffe04, stuart94}. In particular, we can find (see for instance \cite{strachan92,mantonSutcliffe04,Schiff91}) equivariant solutions $(\Phi[Q],A[Q])$ of the form
\begin{equation}
    \Phi[Q](r,\theta)=Q(r)e^{im\theta},\qquad A_r[Q]=0,\qquad A_\theta[Q](r,\theta)\equiv A_\theta[Q](r).
\end{equation}
The last equation is simply the equivariant condition that $A_\theta[Q]$ is independent of $\theta$. Here the vanishing of $A_r[Q]$ is a gauge choice which we refer to as the Coulomb gauge.  The functions $Q$ and $A_\theta[Q]$ satisfy the identities
\begin{equation}
\label{eq:Q}
    \big(\partial_r+\frac{A_\theta[Q]-m}{\sh r}\big)Q=0,\qquad \frac{1}{\sh r}\partial_r A_\theta[Q]=\frac{1}{2}(1-|Q|^2).
\end{equation}

Returning to the time-dependent equation \eqref{eq:Mantonintro1}, note that any solution of \eqref{eq:Bogomolny1} with $A_0\equiv0$ constitutes a solution of \eqref{eq:Mantonintro1}. To consider time-dependent solutions we introduce the notation (note that $D_+=\overline{\partial}_A)$ 
\begin{equation}
    D_+:=e^{i\theta}(D_r+\frac{i}{\sh r}D_\theta),\qquad D_{+}^\ast:=e^{-i\theta}(D_r^{\ast}+\frac{1}{\sh r}+\frac{i}{\sh r}D_\theta),
\end{equation}
where $D_r=\partial_r-iA_r$, $D_{\theta}=\partial_{\theta}-iA_{\theta}$, and $D_r^\ast=\partial_r^\ast+iA_r$ with $\partial_r^\ast=-\partial_r-\coth r$.
It follows that $D_+^\ast D_+=-\Delta_A-B$ where $\Delta_A:=D_r^2+\coth(r)D_r+\frac{1}{\sh^2(r)}D_{\theta}^2.$ Therefore the first equation in \eqref{eq:Mantonintro1} can be written as (here we have also used the third equation in \eqref{eq:Mantonintro1})
\begin{equation}
    iD_t\Phi-\frac{1}{2}D_{+}^\ast D_{+}\Phi=0.
\end{equation}

We now restrict attention to equivariant solutions and impose the Coulomb gauge. That is, we consider solutions of the form
\begin{equation}
\label{eq:equiv-gauge}
    \Phi=e^{im\theta}\phi(t,r), \qquad A_0\equiv A_0(t,r),\qquad A_\theta\equiv A_\theta(t,r),\qquad A_r=0.
\end{equation}
Solutions of this form are called equivariant of degree $m$. Under these assumptions, and in view of our earlier observations, equation \eqref{eq:Mantonintro1} can be written as
\begin{align}
 \label{GL-g-2}
   & i \,  \, D_t \Phi - \frac 12  D_+^{\ast} D_+ \Phi  =0  \\
    \label{eq_E_r_2}
 & \partial_r B - 2   \partial_r A_0 = \frac{1}{\sh(r)} \im \left( \Phi \overline{D_{\theta} \Phi} \right)   \\
  \label{eq_E_theta_2}
 &  2  \partial_0 A_{\theta} =-\sh(r) \im \left( \Phi \overline{D_{r} \Phi} \right)   \\ 
    \label{eq_B_2}
& B=\frac{1}{\sh(r)} \partial_r A_{\theta}= \frac{1}{2}(1- |\Phi|^2)
\end{align}
In fact, the third equation is a consequence of the other equations if it is initially satisfied. See \cite{manton1997}. Note that in view of equations \eqref{eq_E_r_2} and \eqref{eq_B_2}, $A_\theta$ and $A_0$ can be obtained from $\Phi$ by integration in $r$ for each fixed time. 

As mentioned earlier, our interest in this work is to prove the asymptotic stability of $(Q,A_\theta[Q])$ as a solution of this system. For this it is natural to linearize the system about $(Q,A_\theta[Q])$. This linearization is carried out in detail in Section~\ref{sec:Proof-theo}. The result can be summarized as follows. Let (recall that we are using the Coulomb gauge $A_r\equiv0$)
\begin{align*}
    \Phi(t,r,\theta)=e^{im\theta}\big(Q(r)+\epsilon(t,r)\big),\qquad A_\theta(t,r)=A_\theta[Q](r)+a_\theta(t,r).
\end{align*}
Define the pair of operators 
\begin{align*}
  &L_{Q} (\varepsilon) = \partial_r \varepsilon + \frac{A_{\theta}[Q]-m}{\sh(r)} \varepsilon - Q B_Q(\varepsilon), \\
  &B_Q(\varepsilon)= \frac{1}{\sh(r)} \int_0^r \re(Q\varepsilon) \sh(s) ds,
\end{align*}
and
\begin{align*}
    &L^{\ast}_Q =  \partial_r^{\ast} + \frac{A_{\theta}[Q]- m }{\sh(r)} - B_Q^{\ast}(Q \cdot) , \\ 
    &B_Q^{\ast} f = Q \int_r^{\infty} \re(f) ds . 
\end{align*}
These operators are formal adjoints with respect to the real inner product $$(u,v)= \re \int u \bar{v} \sh(r) dr.$$
The desired linearized equation can then be written as
\begin{equation}\label{eq:vareplinintro1}
    i\partial_t\varepsilon-\frac{1}{2}L_Q^\ast L_Q\epsilon=G(\epsilon),
\end{equation}
where $G(\epsilon)$ denotes the nonlinearity and is given in \eqref{eq:N-epsi}. In   Appendix \ref{app:LWP}, we prove that the equation \eqref{eq:vareplinintro1} is locally well-posed in $H^1_m$. See below for the definition of $H^1_m.$

Note that  $A_0$ and $a_\theta$ can be written as non-local expressions in $\varepsilon,$ involving both linear and nonlinear terms. The linear part appears in the nonlocal term of the linearized operator, and their nonlinear contributions are included in the nonlinearity. The super-symmetric factorization $L_Q^\ast L_Q$ is a manifestation of the Bogomolny  structure of the self-dual equations. However, this comes with some caveats: $L_Q$ and $L_Q^\ast$ are \emph{non-local}, \emph{are not complex linear}, and are formal adjoints only with respect to the \emph{real inner product} $(\cdot,\cdot)$ introduce above.   

\subsection{The main result and outline of the argument}
To state our main result we introduce the appropriate function spaces. By $H^k$ we mean the usual Sobolev space on $\mathbb{H}^2$. On $I \times \mathbb{H}^2$, where $I \subset \R,$ we define the Strichartz norm 
$$\left\| \Phi \right\|_{\mathcal{S}(I)}:=\sup_{(p,q)\; \text{admissible}}  \left\| \Phi \right\|_{L^{p} (I;L^{q}(\mathbb{H}^2))}, $$
where $(p,q)$ are admissible pairs if 
\begin{align*}
  \bigg\{ p,q \in (2,\infty) \; \big| \;  1-\frac{2}{q}  \leq  \frac{2}{p} \bigg\} \bigcup \bigg\{ (p,q)=(\infty,2) \bigg\} .\\
\end{align*}
We denote the dual Strichartz space by  $\mathcal{N}(I)=\mathcal{S}(I)^{\ast}.$ Notice that, by construction we have 
\begin{align*}
   \left\| F \right\|_{\mathcal{N}(I)} \leq  \left\| F \right\|_{L^{p^{\prime}} (I;L^{q^{\prime}}(\HH^2)) },
\end{align*}
where $(p,q)$ are admissible pairs. It is convenient to use the following norms for higher regularities,
\begin{align*}
  \left\| f \right\|_{\mathcal{S}^1(I)}:= \left\| f \right\|_{\mathcal{S}(I)} + \left\| \nabla f \right\|_{\mathcal{S}(I)} ,  
\end{align*}
and 
\begin{align*}
  \left\| f \right\|_{\mathcal{N}^1(I)}:= \left\| f \right\|_{\mathcal{N}(I)} + \left\| \nabla f \right\|_{\mathcal{N}(I)} .  
\end{align*}
Note that since the Laplacian in $\mathbb{H}^2$ has spectrum $[-\frac{1}{4},\infty)$ (see for instance \cite{Helgason94}), the undifferentiated term is in fact redundant.  

By a slight abuse of notation we also write $\|F\|_{L^p}=\Big(\int_0^\infty |F(r)|^p\sh r d r\Big)^{\frac{1}{p}}$ whenever $F$ is a radial function. In general, if $F$ is a radial function we define 
\begin{align*}
  \left\| F \right\|_{X_m}:= \left\| e^{im \theta} F \right\|_X,
\end{align*}
where $X$ can be any Sobolev, Strichartz or dual Strichartz norm. For instance
\begin{align*}
    \|F\|_{H^1_m}\simeq \|\partial_rF\|_{L^2}+\left\| \frac{F}{\sh r}\right\|_{L^2}.
\end{align*}

We are now ready to state the main result of this work, which is the stability of $(Q,A_\theta[Q])$.

\begin{theorem}
\label{main-theo}
There exists $\delta>0$ such that if $\varepsilon_0$ is a radial function with $\|\varepsilon_0\|_{H^1_m}\leq \delta$ then \eqref{eq:vareplinintro1} has a global solution $\varepsilon(t)$ with $\varepsilon(0)=\varepsilon_0$ and $\|\varepsilon\|_{\mathcal{S}^1_m(\mathbb{R})}\lesssim \delta$.
\end{theorem}

\begin{remark}
We also prove local existence for \eqref{eq:vareplinintro1} with arbitrary data in $H^1_m$ in Appendix~\ref{app:LWP}.
\end{remark}

\subsubsection*{Outline of the argument}
The main difficulty of the proof of the main theorem is that the operators $L_Q$ and $L_Q^\ast$ in \eqref{eq:vareplinintro1} are non-local and more importantly are not complex linear. Therefore, instead of working with equation~\eqref{eq:vareplinintro1} we work with a related equation which is obtained by applying a (nonlinear) Darboux transform to the original equation~\eqref{GL-g-2}. Specifically, we differentiate equation \eqref{GL-g-2} with respect to $D_+$ and view $D_+\Phi$ as the new unknown. Note that if $D^Q_+$ denotes the covariant derivative with respect to $A[Q]$, then in view of~\eqref{eq:Bogomolny1} we have $D_+^QQ=0$. Therefore $D_+\Phi$ can be viewed as a new linearized unknown. If $\Phi$ is equivariant of degree $m$, then $D_+\Phi$ is equivariant of degree $m+1$ and can be written as $e^{i(m+1)\theta}\varepsilon_1(t,r)$. The new variable $\varepsilon_1$ will now satisfy a more favorable equation of the form
\begin{align*}
    i\partial_t\varepsilon_1-\frac{1}{2}\mathcal{R}_Q\varepsilon_1=N(\varepsilon_1),
\end{align*}
for a suitable nonlinearity $N$ involving $\varepsilon$ and $\varepsilon_1$. The differential operator $\mathcal{R}_Q$ is a standard Schr\"odinger operator of the form $\mathcal{R}_Q=-\Delta+V+1$ with $V$ an exponentially decaying potential. Moreover, $\mathcal{R}_Q$ admits a factorization of the form $\mathcal{R}_Q=A_Q^\ast A_Q-1$ where $A_Q$ and $A_Q^\ast$ are local, complex 
linear, first order differential operators which are formal adjoints with respect to the complex inner product $\langle u,v\rangle=\int_0^\infty u(r) \overline{v(r)}\sh r dr$. The favorable structure of the equations after applying the nonlinear Darboux transform, is similar to the related structure for the classical Schr\"odinger-Chern-Simons system studied in \cite{kimKwonwOh20,kim2Kwon23,kimKwonOh22}. We will show that if $\|\varepsilon\|_{H^1_m}$ is small, then $\|\varepsilon_1\|_{L^2}$ and $\|\varepsilon\|_{H^1_m}$ are equivalent in size, see Lemmas~\ref{lem:epH1mtoep1} and~\ref{lem:epLptoepH1m}. \\

 The most delicate aspect of nonlinear analysis is  estimating the non-local terms $A_0$, $a_\theta$, and $\varepsilon$ in terms of $\varepsilon_1,$  which are most challenging near $r=0.$ The first step is to derive a system of coupled nonlinear elliptic equations relating the original variable $\varepsilon$ to $\varepsilon_1$ in terms of the  linearized  potential $a_{\theta}$. See Section \ref{sec:Prop-proof} for the derivation of these equations.  This formulation naturally leads to the spectral analysis of the operator $\mathcal{H} = -\Delta_{\HH^2} + Q^2$. The spectral properties of $\mathcal{H}$ will, in turn, allow us to invert the associated elliptic operator and construct the solution to the nonlinear elliptic equation via a Lyapunov–Perron type argument. In particular, we write the integral representation of the solution via the Green's kernel. With this framework in place, we derive suitable bounds for the associated Green's kernel to estimate $\varepsilon$ in terms of $\varepsilon_1$ and $a_{\theta}$.  This step relies on precise estimates of the $L^p$-norms for small and large $r,$ using the integral representation of the solution, combined with the asymptotic behavior of the fundamental system for the elliptic equation. The bounds on $A_0$ and $a_\theta$ are then obtained by writing them as explicit integrals in terms of $\varepsilon$ and $\varepsilon_1$. The remainder of the argument consists of proving the full range of Strichartz estimates for the linearized operator for $\varepsilon_1$ (see \cite[Corollary 1.19]{LLOS23}) and estimating the nonlinearity in the dual Strichartz norm. Note that while the non-linearity contains non-localized quadratic terms, the larger range of admissible norms on $\mathbb{H}^2$ compared with $\mathbb{R}^2$ allow us to treat these by Strichartz estimates. \\

\textbf{Notation: }
We denote by $C>0$ an absolute constant that may change from line to line. For non-negative $f$ and $g$ we write $f \lesssim  g,$ if $f \leq C g,$ and if $f \ll g$ to indicate that the implicit constant is small. We use the notation $f=O(g)$ if $|f| \lesssim g. $

\section{Derivation of the equations and analysis of the linearized operators}
\label{sec:Proof-theo}
In this section we first derive the linearized equation for $\varepsilon$. Strictly speaking, this equation is not needed for the proof of the main theorem and is derived here for completeness. More relevant is the equation for the Darboux-transformed variable $\varepsilon_1$. The derivation of this equation, and the elliptic equation relating $\varepsilon$ and $\varepsilon_1$. Deriving these equations is the main content of this section. We then end by a preliminary analysis of the resulting linearized operators.

\subsection{Derivation of the equations}

First, we linearize about the $m$-degree vortex. We decompose the m-equivariant solution $(\Phi,A_{\theta})$ to  \eqref{GL-g-2}-\eqref{eq_B_2} as 

\begin{equation}
\label{eq:decom1}
\Phi= e^{im\theta } \phi= e^{im\theta } ( Q(r)+ \varepsilon(t,r)), \qquad \qquad  A_{\theta}=A_{\theta}[Q] + a_{\theta} ,
\end{equation}
where $(Q,A_{\theta}[Q])$ satisfy
\begin{equation}
\label{Q-equation}
      (\partial_r +  \frac{A_{\theta}[Q] -m } {\sh(r)} )   Q = 0 \qquad \text{and} \qquad 
\frac{1}{\sh(r)}  \partial_r A_{\theta}[Q]= \frac 12 (1-|Q|^2) 
\end{equation}
By \eqref{eq_B_2}, we have 
\begin{align}
\label{Linear-A} 
   A_{\theta}  &=A_{\theta}[Q] + a_{\theta}
\quad \text{where} \quad  a_{\theta}=- \int_{0}^r \re(Q \varepsilon)\sh(s) ds- \frac{1}{2} \int_{0}^r |\varepsilon|^2 \sh(s) ds 
\end{align}

If $\Phi=e^{im\theta } \phi$ is $m$-equivariant then $D_{+}  \Phi$ is $(m+1)$-equivariant and $D_{+}^{\ast}  \Phi$ is $(m-1)$-equivariant, i.e., 
\begin{align*}
D_{+} (e^{ i m \theta}  \phi)  
&=  e^{ i  (m+1) \theta} \widetilde{D}_{+,m}(\phi) \quad \text{where } \qquad  \widetilde{D}_{+,m}(\phi)=  \partial_r + \frac{A_{\theta} -m }{\sh(r)} \\
D_{+}^{\ast} (e^{ i  m \theta}  \phi) 
&= e^{ i  (m-1) \theta} \widetilde{D}_{+,m}^{\ast}(\phi) \quad \text{where } \qquad  \widetilde{D}_{+,m}^{\ast}(\phi)= \partial_r^{\ast} + \frac{A_{\theta} -m+1   }{\sh(r)} 
   \end{align*}
Substituting the decomposition \eqref{eq:decom1} into the equation \eqref{GL-g-2}, we obtain 
\begin{align}
\label{eq:varp4}
    i \partial_t \varepsilon + A_0 Q -\frac{1}{2} \widetilde{D}_{+,m+1}^{\ast} D_{+,m} (Q+\varepsilon) = -A_0 \varepsilon
\end{align}
Since $\phi=Q+\varepsilon,$ by \eqref{Q-equation} we have 
\begin{align*}
\widetilde{D}_{+,m}(\phi)=\partial_r  \varepsilon  + \frac{A_{\theta}[Q]-m}{\sh(r) } \varepsilon  + \frac{a_{\theta}}{\sh(r)} Q + \frac{a_{\theta}}{\sh(r)} \varepsilon.
\end{align*}
Note that by \eqref{Linear-A}, we have $\frac{a_{\theta}}{\sh(r)} \varepsilon$ is a higher order term and 
\begin{align*}
    \frac{a_{\theta}}{\sh(r)}Q &= -\frac{Q}{\sh(r)}\int_0^r \re(Q \varepsilon) \sh(s) ds  -\frac{Q}{2\sh(r)}\int_0^r |\varepsilon|^2 \sh(s) ds  \\
    &:=-B_Q(\varepsilon)-\frac{Q}{2\sh(r)}\int_0^r |\varepsilon|^2 \sh(s) ds.
\end{align*}
Therefore, we have 
\begin{align*}
\widetilde{D}_{+,m}(\phi)=L_Q \varepsilon + \frac{a_{\theta}}{\sh(r)} \varepsilon -\frac{Q}{2\sh(r)}\int_0^r |\varepsilon|^2 \sh(s) ds, 
\end{align*}
where
\begin{align*}
L_Q:=\partial_r  \varepsilon  + \frac{A_{\theta}[Q]-m}{\sh(r) } \varepsilon  -B_Q(\varepsilon),\qquad
    B_Q(\varepsilon):=\frac{Q}{\sh(r)}\int_0^r \re(Q \varepsilon) \sh(s) ds .
\end{align*}
One can check the formal adjoint of $L_Q$ with respect to the real inner product $\langle u , v \rangle=\re \int u \bar{v} \sh(r) dr,$ is 
\begin{align*}
    &L^{\ast}_Q f =  \partial_r^{\ast}f + \frac{A_{\theta}[Q]- m }{\sh(r)}f - B_Q^{\ast}(Q f) , \\ 
    &B_Q^{\ast} f = Q \int_r^{\infty} \re(f) ds . 
\end{align*}
Notice that if $f=\widetilde{D}_{+,m} \phi $ then 
\begin{align*}
 \widetilde{D}_{+,m+1}^{\ast}f-  \phi \int_r^{\infty} \re(\overline{\phi} f) \sh(s) ds = L^{\ast}_Q f    -\varepsilon\int_r^\infty \re((Q+\overline{\varepsilon})f)ds-Q\int_r^\infty \re(\overline{\varepsilon} f)ds+\frac{a_\theta}{\sh r} f.
\end{align*}
Moreover, by \eqref{eq_E_r_2} and \eqref{eq_B_2}, we have $\partial_r A_0=-\frac{1}{2} \re ( \overline{\phi} \widetilde{D}_{+} \phi) .$ Integrating from $r$ to $\infty$  we obtain
\begin{align*}
  A_0 &= \frac{1}{2 } \int_r^{\infty}  \re ( \overline{\phi} f) ds 
\end{align*}
Then in view of \eqref{eq:varp4}, we obtain 
\begin{align}\label{eq:varepsilon-main1}
    i \partial_t \varepsilon -\frac{1}{2} {L}_Q^{\ast} {L}_Q \varepsilon = G(\varepsilon),
\end{align}
where
\begin{align} \label{eq:N-epsi}
\begin{split}
    G(\varepsilon)&
    = -A_0 \varepsilon  + \frac{1}{2} (Q\re(\varepsilon)+\frac{1}{2}|\varepsilon|^2) \varepsilon +\frac{1}{4} |\varepsilon|^2 Q + \frac{a_{\theta}}{\sh(r)} \frac{A_{\theta}[Q]-m}{\sh(r)}  \varepsilon \\
    &  + \frac{1}{2} \bigg(\frac{a_{\theta}}{\sh(r)} \bigg)^2 (Q+ \varepsilon) - Q  \frac{A_{\theta}[Q]-m}{\sh(r)} \frac{1}{2\sh(r)}  \int_0^r |\varepsilon|^2 \sh(s) ds  \\
    &  -\frac{Q}{2} \int_r^{\infty} \re \bigg( (\partial_s \varepsilon + \frac{A_{\theta}[Q]-m}{\sh(s)} \varepsilon + \frac{a_{\theta}}{\sh(s)} \varepsilon  ) \bar{\varepsilon}   \bigg) ds \\
    & + \frac{Q}{2} \int_r^{\infty} \frac{Q}{2\sh(s)} \int_0^s |\varepsilon|^2 \sh(\sigma) d\sigma ds.
\end{split}
\end{align}

As mentioned earlier, the linearized operator is not complex linear and contains non-local terms. Therefore, instead of working with the equation for $\varepsilon,$ we transform the problem into an $(m+1)$-equivariant equation. Specifically, we differentiate the equation with $D_+$ and commute $D_+ $ with $D_t$ and $D_+^{\ast}$. This leads to a self-adjoint linearized operators with respect to the standard complex inner product. The details are as follows.\\

Denote by $\varepsilon_1 := \widetilde{D}_{+,m}( \phi )=\widetilde{D}_{+,m}( Q+\varepsilon ).$ We apply $D_{+}$ to \eqref{GL-g-2}. Note that using \eqref{eq:equiv-gauge}, \eqref{eq_E_r_2}, \eqref{eq_E_theta_2}, and \eqref{eq_B_2} we get
\begin{align}\label{eq:commsforvarep1_1}
\begin{split}
 [D_{+},D_t] \Phi &=   - \frac{1}{2 i } |\Phi|^2 D_+ \Phi   \\
[D_+,D_+^{\ast}] D_+\Phi &=  - 2 \,i  \frac{\ch(r) -1 }{\sh^2(r)} D_{\theta} D_+ \Phi
- 2 \frac{\ch(r)-1}{\sh^2(r)}D_+\Phi 
 +2 BD_+\Phi  \\
  D_{\theta} \left(e^{ i (m+1)\theta } \varepsilon_1 \right) &=- i  \left( A_{\theta} - m - 1 \right) e^{i (m+1) \theta}  \varepsilon_1.
\end{split}
\end{align}
Indeed first note that in view of \eqref{eq:equiv-gauge} equation \eqref{eq_E_theta_2} can be written as
\begin{align*}
    \frac{1}{\sh r}\partial_0A_\theta=-\frac{1}{2}\im(\Phi \overline{D_+\Phi}).
\end{align*}
Using also \eqref{eq_E_r_2}, it follows that
\begin{align*}
    [D_+,D_t]\Phi=-(i\partial_r A_0-\frac{1}{\sh r}\partial_0A_\theta)\Phi=-\frac{1}{2}|\Phi|^2D_+\Phi,
\end{align*}
as claimed. A similar computation using $D_+=e^{i\theta}(D_r+\frac{i}{\sh r}D_\theta)$ and $D_r^\ast=e^{-i\theta}(-D_r-\coth r-\frac{1}{\sh r}+\frac{i}{\sh r}D_\theta)$ gives the commutator identity for $[D_+,D_+^\ast]$. Using the first commutator identity we get
\begin{align*}
    iD_t D_+\Phi-\frac{1}{2}D_+D_+^\ast D_+\Phi-\frac{1}{2}|\Phi|^2D_+\Phi=0.
\end{align*}
Writing $D_+\Phi=e^{i(m+1)\theta}\varepsilon_1$ and using the last two identities in \eqref{eq:commsforvarep1_1}, a careful computation yields
\begin{align}
\label{eq-1-epsilon1}
     i \,  \,D_t  \varepsilon_1   - \frac{1}{2}    \widetilde{D}^{\ast}_{+,m+2} \widetilde{D}_{+,m+1}  \varepsilon_1  +  \frac{\ch(r) -1 }{\sh(r)}\frac{(A_{\theta}-m)}{\sh(r)} \varepsilon_1 -\frac{1}{2}  \varepsilon_1    =0  .
\end{align}

Using the fact that 
\begin{align*}
\widetilde{D}^{\ast}_{+,m+2} \widetilde{D}_{+,m+1} \varepsilon_1= \partial_r^{\ast} \partial_r \varepsilon_1 - \frac{1}{\sh(r)} \partial_r A_{\theta} + \left( \frac{A_{\theta}-m-1}{\sh(r)} \right)^2 \varepsilon_1
\end{align*}
we rewrite, the equation \eqref{eq-1-epsilon1} as 
\begin{multline}
 i \,  \,\partial_t \varepsilon_1 +  A_0 \varepsilon_1   - \frac 12  \bigg(  \partial_r^{\ast} \partial_r \varepsilon_1 + \varepsilon_1 \bigg(   \frac{1}{\sh^2(r)}- \frac{1}{\sh(r)}  \partial_r A_{\theta} - 2 \coth(r) \frac{A_{\theta} -m } {\sh(r)} + \left( \frac{A_\theta - m}{\sh(r)} \right)^2 \bigg)  \bigg)-\frac{1}{2} \varepsilon_1=0     
\end{multline}

Using \eqref{eq_B_2}, \eqref{Q-equation} and the fact that $A_{\theta}=A_{\theta}[Q] + a_{\theta}, $ we have 
\begin{align*}
\frac{1}{\sh(r)} \partial_r A_\theta = \frac{1}{2}(1-|Q+\varepsilon|^2) = \frac{1}{\sh(r)} \partial_r A_\theta[Q] - \re(Q \varepsilon) -\frac{1}{2} |\varepsilon|^2 .
\end{align*}

Thus, we obtain the evolution equation
\begin{align}
    \label{eq-2-epsilon1}
 i \,  \,\partial_t \varepsilon_1    - \frac{1}{2} \mathcal{R}_{Q} \varepsilon_1 =N (\varepsilon_1),
\end{align}
where 
\begin{align*} 
\mathcal{R}_{Q} \varepsilon_1 &= (\partial_r^{\ast} \partial_r +1) \varepsilon_1 +  \varepsilon_1 \bigg(    \frac{1}{\sh^2(r)}- \frac{1}{\sh(r)}  \partial_r A_{\theta}[Q]  + 2 \coth(r) \frac{m -A_{\theta}[Q] } {\sh(r)}  + \left( \frac{A_\theta[Q] - m}{\sh(r)} \right)^2 \bigg) .
\end{align*}
and 
\begin{align*}
   N (\varepsilon_1) &= -  A_0  \varepsilon_1   +\frac{\varepsilon_1}{2} \bigg( \re(Q \varepsilon) + \frac{1}{2} |\varepsilon|^2 -  2 \coth(r)  \frac{a_{\theta}}{\sh(r)}  + 2 \left( \frac{ a_{\theta}}{\sh(r)} \right)\left(  \frac{A_{\theta}[Q]  - m }{\sh(r)}\right)+ \frac{a_{\theta}^2}{\sh^2(r)}   \bigg)    .
\end{align*}
The linear operator $\mathcal{R}_Q$ is now a self-adjoint Schr\"odinger operator with respect to the usual inner product, and the equation is \eqref{eq-2-epsilon1} is the desired Darboux-transformed equation.
\begin{remark}
    Note that introducing the operators
    \begin{align*}
        A_Q &:= \partial_r + \frac{A_{\theta}[Q]-m}{\sh(r)}-\coth r, \qquad 
        A_Q^{\ast} := \partial_r^{\ast} + \frac{A_{\theta}[Q]-m}{\sh(r)}-\coth r,
    \end{align*}
    we can write $\mathcal{R}_Q=A_Q^\ast A_Q -1$ so that 
    \begin{equation}
\label{eq-2-epsilon1-1}
 i \,  \,\partial_t \varepsilon_1    - \frac 12(A_Q^{\ast}A_Q -1)\varepsilon_1   =N (\varepsilon_1)
\end{equation}
\end{remark}

We now turn to deriving the equation relating $\varepsilon$ and $\varepsilon_1$ in terms of the  linearized electromagnetic potential $a_{\theta}.$ From the decomposition of the potential $A$ in \eqref{Linear-A}, we deduce

\begin{align}
\label{eq_a_theta-esp}
    -\partial_r^{\ast} (  \frac{a_{\theta}}{\sh(r)} ) +Q\,  \re(\varepsilon) = - \frac{1}{2} |\varepsilon|^2 .
\end{align}
Recall that  
\begin{align*}
  \varepsilon_1 = \widetilde{D}_{+} ( Q+ \varepsilon)
  =  \partial_r \varepsilon + \frac{A_{\theta}[Q]  - m }{\sh(r)} \varepsilon + \frac{a_{\theta}}{\sh(r)} (Q+ \varepsilon )
\end{align*}
Using \eqref{Q-equation} with the definition of $\varepsilon_1,$ we have

\begin{align}
\label{partial_eps_equ}
\partial_r \big( \frac{\varepsilon}{Q} \big) = \frac{1}{Q} \bigg(\partial_r \varepsilon - \frac{\partial_r Q}{Q} \varepsilon \bigg) = \frac{1}{Q} \left( \partial_r \varepsilon +  \frac{A_{\theta}[Q]  - m }{\sh(r)} \varepsilon    \right) =\frac{1}{Q} \varepsilon_1  -\frac{1}{Q}  \frac{a_{\theta}}{\sh(r)} (Q+ \varepsilon)
\end{align}
Multiplying  by $Q$ and taking the real part, and using the fact that $a_{\theta}$ and $Q$ are real-valued functions, yields
 \begin{align}
  \label{eq_1_epsi--epsi_1}
 \partial_r \big( \frac{\re(\varepsilon)}{Q} \big) Q +  \frac{a_{\theta}}{\sh(r)} Q =  \re(\varepsilon_1)  - \frac{a_{\theta}}{\sh(r)} \re(\varepsilon) 
\end{align}
Denote by $f(r):=\frac{\re(\varepsilon(r))}{Q(r)}$ and $g(r):=\frac{a_{\theta}(r)}{\sh(r)}$. We rewrite \eqref{eq_a_theta-esp} and \eqref{eq_1_epsi--epsi_1} as
\begin{equation}
\begin{cases}
\sh(r) \partial_r f + \sh(r) g = \sh(r) \frac{\re(\varepsilon_1)}{Q} - \sh(r) \frac{a_{\theta}}{\sh(r)} \frac{\re(\varepsilon)}{Q}  \\
-\partial_r^{\ast} g = -  Q^2 \, f - \frac{1}{2} |\varepsilon|^2 
\end{cases}
\end{equation} 
Taking the derivative of the first equation with respect to $r$, we obtain
\begin{align*}
\partial_r^2 f+ \coth(r) f + \partial_r  g + \coth(r) g &= \frac{1}{\sh(r)}  \partial_r \left( \sh(r) \frac{\re(\varepsilon_1)}{Q} \right) - \partial_r \left( \sh(r) \frac{a_{\theta}}{\sh(r)} \frac{\re(\varepsilon)}{Q}  \right)
\end{align*}
which, upon using the second equation, yields
\begin{equation}   
\label{eq:epsilon-epsilon_1-1}
  \mathcal{H}f =F(\varepsilon, \varepsilon_1), \qquad \text{where} \quad \mathcal{H}= -\Delta_{{\mathbb{H}^2}} + Q^2 ,\qquad f=\frac{\re(\varepsilon)}{Q},
\end{equation}
 \begin{align*}
F(\varepsilon, \varepsilon_1)&=          -\frac{1}{\sh(r)}  \partial_r\left(\sh(r) \frac{\re(\varepsilon_1)}{Q} \right) + \frac{1}{\sh(r)}  \partial_r \left( \sh(r) \frac{a_{\theta}}{\sh(r)} \frac{\re(\varepsilon)}{Q} \right) -\frac{1}{2} |\varepsilon|^2 
\end{align*}
This is the main equation relating $\re\, \varepsilon$ with $\varepsilon_1$. Note that $\varepsilon$ itself appears nonlinearly on the right-hand side of this equation. For the  imaginary part, recall that by \eqref{partial_eps_equ} we have
\begin{align*}
\partial_r \big( \frac{\varepsilon}{Q} \big) =\frac{1}{Q} \varepsilon_1  +\frac{1}{Q}  \frac{a_{\theta}}{\sh(r)} (Q+ \varepsilon).
\end{align*}
Integrating between $r$ and $\infty$, we obtain
\begin{align*}
   \varepsilon(r) = - Q(r) \int_r^{\infty}   \frac{\varepsilon_1}{Q(s)}  ds - Q(r) \int_r^{\infty}  \frac{a_{\theta}}{\sh(s)}  ds - Q(r) \int_r^{\infty}   \frac{1}{Q(s)}  \frac{a_{\theta}(s)}{\sh(s)}  \varepsilon(s) ds.
\end{align*}  
Taking the imaginary part and using the fact that $a_{\theta}$ is a real-valued function by \eqref{Linear-A}, we get
\begin{align}
\label{Im-eq_eps}
 \im(\varepsilon(r))=  Q(r) \int_r^{\infty}   \frac{ \im(\varepsilon_1(s))}{Q(s)}  ds - Q(r) \int_r^{\infty}   \frac{1}{Q(s)}  \frac{a_{\theta}(s)}{\sh(s)}  \im(\varepsilon(s)) ds.        
\end{align}
This is the nonlinear equation relating $\im\,\varepsilon$ to $\varepsilon_1$.

Finally, to estimate $\frac{a_\theta}{\sh r}$ and $A_0$ we use the following equations. For $\frac{a_\theta}{\sh r}$, recall that by \eqref{Linear-A}, we have 
\begin{align}
\label{equ-integral-a_theta/sinh}
\frac{a_{\theta}(r) }{\sh(r)} =  - \frac{1}{2\sh(r)} \int_0^r |\varepsilon|^2 \sh(s) ds  - \frac{1}{\sh(r)}\int_0^r Q \re(\varepsilon) \sh(s) ds. 
\end{align}
Next, notice that by \eqref{eq_E_r_2} and \eqref{eq_B_2}, we have $\partial_r A_0=-\frac{1}{2} \re ( \bar{\phi} \title{D}_{+} \phi) .$ Integrating from $r$ to $\infty$ and using the decomposition of $\phi$ together with the definition of $\varepsilon_1$ we obtain
\begin{align}\label{eq:A0toep1}
  A_0 &= \frac{1}{2 } \int_r^{\infty}   \re(\varepsilon_1 ( Q + \bar{\varepsilon})) ds = \frac{1}{2 } \int_r^{\infty}( \re(\varepsilon_1) Q + \re(\varepsilon_1 \bar{\varepsilon})) \, ds.
\end{align}

Equations \eqref{eq-2-epsilon1}, \eqref{eq:epsilon-epsilon_1-1}, \eqref{Im-eq_eps}, \eqref{equ-integral-a_theta/sinh}, and \eqref{eq:A0toep1} are the main equations that we will use for the analysis in this paper. 

\subsection{Analysis of the linear operators $\mathcal{R}_Q$ and $\mathcal{H}$}
Our analysis of well-posedness and stability relies on the Strichartz estimates established in  \cite{LLOS23} and \cite{AnkerJeanPierfelice09} for the Schr\"odinger type equations. In \cite{LLOS23}, the authors treated a general class of operators under suitable spectral assumptions. We verify these assumptions for our operator $\mathcal{R}_Q$ in the following lemma.
\begin{lemma}\label{lem:RQspec1}
    The spectrum for the self-adjoint operator $\mathcal{R}_Q$ is purely absolutely continuous and given by $\sigma(\calR_Q)=[\frac{5}{4},\infty).$ In particular, there is no negative spectrum, there are no eigenvalues in the gap
$[0,\frac{5}{4}),$ and the threshold $\frac{5}{4}$ is neither an eigenvalue nor a resonance.
\end{lemma}
\begin{proof}
It well know that the spectrum of the Laplacian on $\HH^2$ is given by $\sigma(\Delta_{\HH^2})=[\frac{1}{4},\infty),$ where $\frac{1}{4}=\left(\frac{d-1}{2} \right)^2 .$ Write $\calR_{Q}= (-\Delta_{\HH^2}+1)+ V(r) ,$ where $V$ is exponentially decaying potential. Since $R_Q$ is a self-adjoint operator then using the Weyl criterion  (see for instance \cite[Theorem {\rm{XIII.14}}]{ReedSimonIV}) one can check that the essential spectrum of the self-adjoint operator $ \calR_{Q}$ is given by $\sigma( \calR_{Q})=[\frac{5}{4},\infty).$\\ 

Next, we prove that $\calR_Q$ has no eigenvalue in $[0,\frac{5}{4}]$. Since $|A_{\theta}[Q]|\leq m$ and $1-Q^2\geq 0,$ then we have $\partial_r V<0.$
Indeed, 

\begin{align*}
   \partial_r V(r) &=  -2 \frac{\ch(r)}{\sh^3(r)} + Q^2 \frac{m-A_{\theta}[Q]}{\sh(r)} - 2 \frac{1}{\sh^2(r)} \frac{m-A_{\theta}[Q]}{\sh(r)}- \coth(r)(1-Q^2) - 2 \coth^2(r) \frac{m-A_{\theta}[Q]}{\sh(r)} \\
   &-  \frac{m-A_{\theta}[Q]}{\sh(r)} ( 1- Q^2) 
   -2 \coth(r) \left( \frac{m-A_{\theta}[Q]}{\sh(r)} \right)^2 < 0 . 
\end{align*}

Let $\beta(r)=\frac{\ch(r)-1}{\sh(r)}$ and $M u(r)=\beta(r) \partial_r u(r) - \partial_r^{\ast} (\beta(r) u(r)).$ Notice that $M^{\ast}=-M,$ and $\beta$ satisfies
\begin{equation}
\label{eq:Propbeta}
   0<\beta<1 , \quad \partial_r \beta >0, \; \;  \text{ and } \; \; \Delta_{\HH^2} \partial_r^{\ast} \beta = 0.
\end{equation}
Assume that $ \calR_{Q}$ has an eigenvalue less than or equal to $\frac{5}{4},$ that is, let $u \in L^2$ such that  $ \calR_{Q} u = \mu^2 u .$ Then we have 
\begin{align*}
0=\langle \mu^2 u, Mu \rangle =\langle \calR_{Q}u, Mu \rangle  = 2 \langle \partial_r \beta \partial_r u, \partial_r u  \rangle + \frac{1}{2}\langle (\Delta \partial_r^{\ast} \beta) u,u  \rangle - \langle \beta (\partial_r V) u,u \rangle
\end{align*}
Using \eqref{eq:Propbeta} with the fact $\partial_r V<0,$ we obtain 
\begin{align*}
    0=\langle \mu^2 u, Mu \rangle  >=2 \langle \partial_r \beta \partial_r u, \partial_r u  \rangle - \langle \beta (\partial_r V) u,u \rangle >0,
\end{align*}
which is impossible. Therefore, the operator $ \calR_{Q}$ has no eigenvalue less than or equal to 
$\frac{5}{4}.$ Next, we show that $\tfrac{5}{4}$ is not a resonance in the sense defined below. By standard ODE techniques, one can check that the possible behavior of $R_Qu = \frac{5}{4} u,$ for $u\in L^{\infty}\backslash L^2 ,$ near $0$ are $\{r^{\frac{1}{2}+m},r^{-1-m} \}$
and near $\infty$ are $\{e^{-\frac{r}{2}}, re^{-\frac{r}{2}}   \}.$ We say $R_Q$ admits a resonance if the subordinate behavior near $0,$ i.e., $r^{\frac{1}{2}+m}$ connects to the subordinate behavior near $\infty,$ i.e., $e^{-\frac{r}{2}}.$ Denote the resonance by $\varphi$ such that $ (-\Delta+1+V) \varphi = \frac{5}{4} \varphi $ or equivalently
\begin{align*}
    (-\Delta-\frac{1}{4}+V) \varphi = 0
\end{align*}
and for some $c>0,$ we have 
\begin{equation}
  \varphi(r):=  \begin{cases}
        c r^{\frac{1}{2}+m}, \quad  \text{as } r \to 0, \\
        c e^{-\frac{r}{2}}, \quad \; \; \, \text{as } r \to \infty.
    \end{cases}
\end{equation}
First notice that one hand, we have  $\int_0^\infty  (-\Delta-\frac{1}{4}+V) \varphi M \varphi \sh(r) dr=0 ,   $ 
and on the other hand, using integration by parts, for some $R>0,$ we have
\begin{align*}
  \int_0^R (-\Delta-\frac{1}{4}+V) \varphi M \varphi \sh(r) dr&= - \beta(R)\sh(R) (\partial_r \varphi(R))^2 + \sh(R) \partial_r \varphi (\partial_r^{\ast} \beta)(R) \\
  &- \frac{1}{2} \varphi^2(R) \sh(R) \partial_r \partial_r^{\ast} \beta (R) + 2 \int_0^R \partial_r \beta (\partial_r \varphi)^2 \sh(r) dr \\
  &- \int_0^r \beta  \partial_r V \psi^2 \sh(r) dr 
\end{align*}
Let $R\to \infty,$ the boundary terms vanishes, indeed, it is equal to 
$-\frac{1}{4} \frac{c^2}{2}-\frac{1}{4} \frac{c^2}{2} + \frac{1}{2} \frac{c^2}{2}=0.$ Since $\beta>0, \partial_r \beta >0, \partial_r V\leq 0$ then we have $  \int_0^{\infty} (-\Delta-\frac{1}{4}+V) \varphi M \varphi \sh(r) dr>0,$ which impossible. Therefore the operator $R_Q$ has no threshold resonance at $\frac{5}{4}$. 
\end{proof}
As a corollary we obtain Strichartz estimates for $\mathcal{R}_Q$.
 
\begin{coro}
\label{Strichartz}
Let $\varepsilon_1(t)$ be a solution of \eqref{eq-2-epsilon1} on any time interval $I \subseteq \R$, with initial condition $\varepsilon_1(0) \in L^2$.  Then for any two
pairs of exponents $(p_1, q_1)$ and $(p_1, q_1)$,  in
\begin{align*}
  \bigg\{ p,q \in (2,\infty) \; \big| \;  \frac{d}{2}-\frac{d}{q}  \leq  \frac{2}{p} \bigg\} \bigcup \bigg\{ (p,q)=(\infty,2) \bigg\}
\end{align*}
we have 
\begin{align*}
   \left\| \varepsilon_1 \right\|_{L^{p_1} (I;L^{q_1}(\mathbb{H}^d))} \lesssim \left\| \varepsilon_1(0) \right\|_{L^2 (\mathbb{H}^d) }  + \left\| N(\varepsilon_1) \right\|_{L^{p_1^{\prime}}(I, L^{q_1^{\prime}})},
\end{align*}
where $\frac{1}{p_1^{\prime}}+\frac{1}{p_1}=1 $ and $\frac{1}{q_1^{\prime}}+\frac{1}{q_1}=1. $ 
\end{coro}
\begin{proof}
    This is a direct consequence of \cite[Corollary 1.19]{LLOS23} and Lemma~\ref{lem:RQspec1}.
\end{proof}
In our analysis we will use the following available exponents $(4,4), (\frac{8}{3},\frac{8}{3})$, and $(4,\frac{8}{3}).$

Next we turn to the elliptic operator $\mathcal{H}=\partial_r^\ast\partial_r+Q^2$. Conjugating $\mathcal{H}$ to the half line, i.e., multiplying by $\sh^{\frac{1}{2}}(r)$, we obtain $H:=\sh^{\frac{1}{2}}r\cdot \mathcal{H}\cdot \sh^{-\frac{1}{2}}r$, with
\begin{equation}
    H = -\partial_r^2 -\frac{1}{4\sh^2(r)} + \frac{5}{4} + (Q^2-1). 
\end{equation}
Thus, by \eqref{eq:epsilon-epsilon_1-1} we get 
\begin{align}
\label{eq:epsilon-epsilon_1-2}
     H \tilde{f}=\tilde{F}(\varepsilon, \varepsilon_1) 
\end{align}
where $\tilde{f}:=\sh^{\frac{1}{2}} (r)f $ and $\tilde{F}(\varepsilon, \varepsilon_1):=\sh^{\frac{1}{2}}(r) F(\varepsilon, \varepsilon_1).$
\begin{lemma}\label{lem:Hfundsyst1}
    Zero is not an eigenvalue of the operator $H$. This operator admits a fundamental system of solutions $\{\phi_0, \phi_\infty\}$, $H\phi_0=H \phi_\infty=0$, that up to multiplicative constants satisfy the asymptotic behaviors 
    \begin{align}
\label{asymp_phi0}
     \phi_0(r)\sim
 \begin{cases}
       r^\frac{1}{2}, &r \longrightarrow 0^{+} ,  \\
      e^{\frac{\sqrt{5}}{2} r }, &r \longrightarrow \infty ,
    \end{cases} 
\end{align}
\begin{align}
\label{asymp_phi-infty}
  \phi_{\infty} (r)\sim
 \begin{cases}
     r^\frac{1}{2} \log(r),  &r \longrightarrow 0^{+} ,\\
      e^{-\frac{\sqrt{5}}{2} r },  &r \longrightarrow \infty.
    \end{cases}  
\end{align}
\end{lemma}
\begin{proof}
Consider the equation $H\psi=0$. Near $r=0,$ the possible asymptotics of $\psi$ are $\{r^{\frac{1}{2} }, r^{\frac{1}{2}} \log(r)  \} $ and for large $r$ the possible asymptotics are $\{e^{-\frac{\sqrt{5}}{2} r},e^{\frac{\sqrt{5}}{2} r} \}$. Indeed, the operators $-\partial_r^2 -\frac{1}{4r^2}$ and $-\partial_r^2 + \frac{5}{4}$ are a good approximation of $H$ near $r=0$ and $r=\infty,$ respectively. The claim about the possible asymptotics then follows from the variation of parameters formula. Let $u =\sh^{\frac{1}{2}}(r)(m-A_{\theta}[Q]), $ then we have $H(u)=\sh^{\frac{1}{2}}(r) \ch(r) (1-Q^2(r))\geq 0. $ Assume by contradiction, that there exists a nontrivial function 
$v \in L^2(0,\infty)$ such that $Hv = 0.$ Without loss of generality, we may suppose that $ v $ is positive in a neighborhood of the origin. Since $ v \in L^2(0,\infty)$, then it must vanish at some point in $[0,\infty]$. Let $0 < r_0 \leq \infty$ be the first value such that $v(r_0)=0.$ Notice that $0< \int_0^{r_0} Hu \, v dx.$ Using integration by parts, one can see that $\int_0^{r_0} Hu \, v dx \leq 0,$ which is impossible. Therefore, the only possible asymptotics for $\phi_0$ and $\phi_{\infty}$ are those in \eqref{asymp_phi0} and \eqref{asymp_phi-infty}.
\end{proof}
\begin{remark}
    We expect that $H$ has an eigenvalue (possibly more) in the gap $(0,\frac{5}{4})$. Indeed, in the case of the hyperbolic plane with curvature $-\frac{1}{2}$, where the soliton has an explicit formula, this can be verified directly. This fact has no bearing on our analysis.
\end{remark}

As a corollary of Lemma~\ref{lem:Hfundsyst1} the solution $\tilde{f}(r)$ to \eqref{eq:epsilon-epsilon_1-2} admits the representation,
\begin{align*}
\tilde{f}(r)= C \left( \phi_{\infty}(r) \int_0^r \phi_0(s)  \tilde{F}(\varepsilon(s), \varepsilon_1(s)) ds + \phi_{0}(r) \int_r^{\infty} \phi_{\infty}(s) \tilde{F}(\varepsilon(s), \varepsilon_1(s)) ds \right)  ,
\end{align*}
where $C\neq0$ is a constant. 
Therefore, in view of the definition of $\tilde{F}(\varepsilon, \varepsilon_1):=\sh^{\frac{1}{2}}(r) F(\varepsilon, \varepsilon_1),$ and $\tilde{f}(r)=\sh^{\frac{1}{2}}(r) \frac{\re(\varepsilon(r))}{Q(r)}$ we have 
\begin{align}
\label{int-equ-eps}
\begin{split}
    \re(\varepsilon(r))   &= C \frac{Q(r)}{\sh(r)^\frac{1}{2}}   \phi_{\infty}(r) \int_0^r \phi_0(s) \sh^\frac{1}{2}(s) F(\varepsilon(s), \varepsilon_1(s)) ds \\  &+  C \frac{Q(r)}{\sh(r)^\frac{1}{2}} \phi_{0}(r) \int_r^{\infty} \phi_{\infty}(s) \sh^\frac{1}{2}(s) F(\varepsilon(s), \varepsilon_1(s)) ds 
\end{split}
\end{align}
The next lemma provides the necessary asymptotics for $\phi_0$ as $r \to 0$, which will be needed for a later estimate using this formula. Note that $\phi_0 \in L^2(0,1)$. Here, we write $H=H_0+ Q^2,$ where $H_0:=- \partial_r^2 + \frac{1}{4}-\frac{1}{4 \sh^2(r)} .$ 
 \begin{lemma}  
Let $\phi_0$ be the bounded solution near $0$ to $H_0 f = 0$. Then, for small $r$, we have:

 \begin{equation}
 \label{Behav_phi_0_small_r}
  \phi_0(r)=\sh(r)^{\frac{1}{2}} + O(r^{2m+2}) \qquad 
 \end{equation} 
 
\end{lemma}
\begin{proof}
One can check that  $\psi_1(r)=\sh^{\frac{1}{2}}(r)$ and $\psi_2(r)=\sh^{\frac{1}{2}}(r) \log(\tha(\frac{r}{2}))$ are two solutions of $H_0 \psi=0.$
Recall that $Q(r)=r^{m} +O(r^{m+1}) $ for small $r$. Using the Green's function, we have 
\begin{align*}
 \phi_0(r) &=\sh^{\frac{1}{2}}(r) + C \int_0^r \left( \psi_1(r) \psi_2(s) - \psi_1(s) \psi_2(r) \right) Q^2(s)   \phi_0(s)ds   \\
 & = \sh^{\frac{1}{2}}(r) +  C \left( r^{\frac{1}{2}} \log(\frac{r}{2}) \int_0^r s^{\frac{1}{2}} s^{2m} ds - r^{\frac{1}{2}} \int_0^r  s^{\frac{1}{2}} \log(\frac{s}{2}) s^{2m} ds  \right) + O(r^{2m+3}) \\
 &= \sh^{\frac{1}{2}}(r) + C \left( r^{2m+2} \log(\frac{r}{2}) - r^{2m+2} \log(\frac{r}{2}) + r^{2m+2}   \right) + O(r^{2m+3}) \\
 &=  \sh^{\frac{1}{2}}(r)  + O(r^{2m+2}) .\qedhere
\end{align*}
\end{proof}

\section{Proof of Theorem~\ref{main-theo}}
\label{sec:Prop-proof}
In this section we prove Theorem~\ref{main-theo}. First we prove Strichartz estimates on $\varepsilon_1$ in Proposition~\ref{main-prop}, in Subsection~\ref{subsec:Mainprop}, under suitable smallness assumptions. The analysis of the Green's function for $H$ from the previous section play a key role in  establishing these estimates. Then, in Subsection~\ref{subsec:Proof-of-Theo}, we use Proposition~\ref{main-prop} to complete the proof of Theorem~\ref{main-theo}.

\subsection{The main proposition}
\label{subsec:Mainprop}
In this section we state and prove the main proposition that gives control of the Strichartz norms of $\varepsilon_1$.
\begin{prop}
\label{main-prop}
    Suppose $\|\varepsilon_1(0)\|_{L^2}+\|\varepsilon(0)\|_{H^1_m}\lesssim \delta$. If $\delta$ is sufficiently small, then on any interval $I$ on which  \eqref{eq:varepsilon-main1} and \eqref{eq-2-epsilon1} are satisfied, 
    \begin{align*}
        \sup_{t\in I}\|\varepsilon(t)\|_{H^1_m}+\|\varepsilon_1\|_{\mathcal{S}(I)}\lesssim \delta.
    \end{align*}
    The implicit constants in this estimate is independent of $I$.
\end{prop}
We will break the proof of the proposition into a number of lemmas. 
First we estimate the $L^p$-norms of $\varepsilon$ in terms of $\varepsilon_1$ for $2 \leq p \leq \infty.$

\begin{lemma} \label{lem:epsi-LP-epsi1L_p}
Let $  2 \leq p,p_1 < \infty$, and $2 \leq q_1 \leq \infty $ 
\begin{align}
     \left\| \varepsilon \right\|_{L^p}  & \lesssim \left\| \varepsilon_1 \right\|_{L^p} + \left\| \frac{a_\theta}{\sh(r)} \right\|_{L^{\infty}}^2  +  \left\| \varepsilon \right\|_{L^{\infty}}^2
\end{align}
\begin{align}
\left\| \varepsilon \right\|_{L^\infty} &\lesssim \left\| \varepsilon_1 \right\|_{L^{p_1}} + \left\| \frac{a_\theta}{\sh(r)} \right\|_{L^{\infty}}^2 +  \left\|\varepsilon_1 \right\|_{L^{q_1}}^2
\end{align}
\end{lemma}
\begin{proof}
We estimate the real and imaginary parts of $\varepsilon$ separately using the integral solution \eqref{int-equ-eps} for the real part, and \eqref{partial_eps_equ} together with the fact that $a_{\theta} $ is real-valued function for the imaginary part. 

\begin{claim}
\label{claim-real-eps-norm}

Let $   2 \leq p,p_1,q_1 < \infty$,  then we have

\begin{align}
\label{estim-real-eps}
     \left\| \re(\varepsilon) \right\|_{L^p}  &\leq \left\| \varepsilon_1 \right\|_{L^p} + \left\| \frac{a_\theta}{\sh(r)} \right\|_{L^{\infty}} \left\| \varepsilon \right\|_{L^p} + 
     \left\| \varepsilon \right\|_{L^p} \left\| \varepsilon \right\|_{L^{\infty}}  
\end{align}

\begin{align}
\label{real-eps-infty-norm}
     \left\| \re(\varepsilon) \right\|_{L^\infty}  &\lesssim \left\| \varepsilon_1 \right\|_{L^{p_1}} + \left\| \frac{a_\theta}{\sh(r)} \right\|_{L^{\infty}} \left\|\varepsilon \right\|_{L^{q_1}} + 
     \left\| \varepsilon \right\|_{L^{\infty}}^2 
\end{align}

\end{claim}
\begin{proof}
 Recall that by \eqref{int-equ-eps}, we have

\begin{align*}
   \re(\varepsilon(r))   &= C \frac{Q(r)}{\sh(r)^\frac{1}{2}}   \phi_{\infty}(r) \int_0^r \phi_0(s) \sh^\frac{1}{2}(s) F(\varepsilon(s), \varepsilon_1(s)) ds \\  &+  C \frac{Q(r)}{\sh(r)^\frac{1}{2}} \phi_{0}(r) \int_r^{\infty} \phi_{\infty}(s) \sh^\frac{1}{2}(s) F(\varepsilon(s), \varepsilon_1(s)) ds  
\end{align*}
where \begin{align*}
\sh^\frac{1}{2}(r) F(\varepsilon, \varepsilon_1)&=          -\frac{1}{\sh^\frac{1}{2}(r)}  \partial_r\left(\sh(r) \frac{\re(\varepsilon_1)}{Q} \right) \\
&+ \frac{1}{\sh^\frac{1}{2}(r)}  \partial_r \left( \sh(r) \frac{a_{\theta}}{\sh(r)} \frac{\re(\varepsilon)}{Q} \right) -\frac{1}{2}\sh^\frac{1}{2}(r) |\varepsilon|^2 
\end{align*}
Denote:
\begin{align}
\label{def-F-1}
F_1(\varepsilon, \varepsilon_1)&=-\frac{1}{\sh^{\frac{1}{2}}(r)}  \partial_r\left(\sh(r) \frac{\re(\varepsilon_1)}{Q} \right) , \\
\label{def-F-2}
F_2(\varepsilon, \varepsilon_1)&=
\frac{1}{\sh^{\frac{1}{2}}(r)}  \partial_r \left( \sh(r) \frac{a_{\theta}}{\sh(r)} \frac{\re(\varepsilon)}{Q} \right), \\
\label{def-F-3}
F_3(\varepsilon, \varepsilon_1)&= -\frac{1}{2}\sh^{\frac{1}{2}}(r) |\varepsilon|^2 . 
\end{align}

Then, we have
\begin{align}
\label{est-real-eps-p-norm}
\int_0^{\infty}    \left| \re(\varepsilon(r)) \right|^p \sh(r) dr &\leq
\int_0^{\infty}   I_1(r) \sh(r) dr +  \int_0^{\infty}   I_2(r) \sh(r) dr +  \int_0^{\infty} I_3(r) \sh(r) dr
\end{align}
where \begin{align*}
 I_1:&=     \frac{Q^p(r)}{\sh^\frac{p}{2}(r)}  \left| \phi_{\infty}(r) \int_0^r \phi_0(s) F_1(\varepsilon(s), \varepsilon_1(s)) ds + \phi_{0}(r) \int_r^{\infty} \phi_{\infty}(s) F_1(\varepsilon(s), \varepsilon_1(s)) ds \right|^p , \\
I_2:&=   \frac{Q^p(r)}{\sh^\frac{p}{2}(r)}  \left| \phi_{\infty}(r) \int_0^r \phi_0(s) F_2(\varepsilon(s), \varepsilon_1(s)) ds +   \phi_{0}(r) \int_r^{\infty} \phi_{\infty}(s) F_2(\varepsilon(s), \varepsilon_1(s)) ds \right|^p , \\
I_3:&=  \frac{Q^p(r)}{\sh^\frac{p}{2}(r)}  \left| \phi_{\infty}(r) \int_0^r \phi_0(s) F_3(\varepsilon(s), \varepsilon_1(s)) ds 
   +   \phi_{0}(r) \int_r^{\infty} \phi_{\infty}(s) F_3(\varepsilon(s), \varepsilon_1(s)) ds \right|^p  .
\end{align*}

\textbf{Step 1: Estimate $I_j$ for r small.} We first estimate $I_1$ for $r \leq \delta,$ for some $\delta >0,$ using the asymptotic \eqref{asymp_phi-infty} for $\phi_{\infty},$ Lemma \ref{Behav_phi_0_small_r}, H\"older's inequality with Schur's Test.

Let $\chi(r)$ be smooth function such that $\chi(r)=1$ if $r \leq \frac{1}{2}$ and $\chi(r)=0$ if $r \geq 1.$ Denote by $\Tilde{\chi}(r) = 1-\chi(r),$ thus $\Tilde{\chi}(r) = 0$ if $r\leq \frac{1}{2}$ and  $\Tilde{\chi}(r) =1$ if $r \geq 1.$ Then we have 

\begin{align*}
 \int_0^{\delta}  I_1(r) \sh(r) dr 
  & \lesssim    \int_0^{\delta}   \frac{Q^p(r)}{\sh^\frac{p}{2}(r)}  \bigg(  \bigg| \phi_{\infty}(r) \int_0^r \phi_0(s) \frac{1}{\sh^{\frac{1}{2}}(s)}  \partial_s\left(\sh(s) \frac{\re(\varepsilon_1)}{Q} \right) ds \\
&  \qquad \qquad \qquad \quad  + \phi_{0}(r) \int_r^{1} \chi(s) \phi_{\infty}(s) \frac{1}{\sh^{\frac{1}{2}}(s)}  \partial_s\left(\sh(s) \frac{\re(\varepsilon_1)}{Q} \right) ds \bigg|^p  \bigg) \sh(r) dr \\
&+ \int_0^{\delta}   \frac{Q^p(r)}{\sh^\frac{p}{2}(r)}  \bigg(  \bigg| \phi_{0}(r) \int_{\frac{1}{2}}^{\infty} \Tilde{\chi}(s) \phi_{\infty}(s) \frac{1}{\sh^{\frac{1}{2}}(s)}  \partial_s\left(\sh(s) \frac{\re(\varepsilon_1)}{Q} \right) ds \bigg|^p  \bigg) \sh(r) dr \\
 & \lesssim    \int_0^{\delta}   \frac{Q^p(r)}{\sh^\frac{p}{2}(r)}  \bigg(  \bigg| \phi_{\infty}(r) \left[\phi_0(s) \sh^{\frac{1}{2}}(r) \frac{\re(\varepsilon_1)}{Q} \right]_{s=0}^{s=r}  \\
 &  \quad \qquad \qquad \qquad- \phi_{\infty}(r) \int_0^r  \partial_s \left( \phi_0(s) \sh^{-\frac{1}{2}}(s) \right) \sh(s) \frac{\re(\varepsilon_1)}{Q} ds \\
&   \quad \qquad \qquad \qquad  +\phi_{0}(r) \left[\chi(s) \phi_{\infty}(s) \sh^{\frac{1}{2}}(r) \frac{\re(\varepsilon_1)}{Q} \right]_{s=r}^{s=1} \\
&  \quad \qquad \qquad \qquad- \phi_{0}(r) \int_r^1  \partial_s \left( \chi(s) \phi_{\infty}(s) \sh^{-\frac{1}{2}}(s) \right) \sh(s) \frac{\re(\varepsilon_1)}{Q} ds \bigg|^p \bigg) \sh(r) dr  \\
&+ \int_0^{\delta}   \frac{Q^p(r)}{\sh^\frac{p}{2}(r)}  \bigg(  \bigg| \phi_{0}(r) \int_{\frac{1}{2}}^{\infty}\Tilde{\chi}(r)  \phi_{\infty}(s) \frac{1}{\sh^{\frac{1}{2}}(s)}  \partial_s\left(\sh(s) \frac{\re(\varepsilon_1)}{Q} \right) ds \bigg|^p  \bigg) \sh(r) dr \\
 & \lesssim   \int_0^{\delta}   \frac{Q^p(r)}{\sh^\frac{p}{2}(r)}    \bigg|   \phi_{\infty}(r) \int_0^r  \partial_s \left( \phi_0(s) \sh^{-\frac{1}{2}}(s) \right) \sh(s) \frac{\re(\varepsilon_1)}{Q} ds \bigg|^p \sh(r) dr \\ 
 &  +\int_0^{\delta}   \frac{Q^p(r)}{\sh^\frac{p}{2}(r)} 
 \bigg|\phi_{0}(r) \int_r^1  \partial_s\left( \chi(s)\right)  \phi_{\infty}(s) \sh^{-\frac{1}{2}}(s)  \sh(s) \frac{\re(\varepsilon_1)}{Q} ds \bigg|^p \sh(r) dr \\
 & + \int_0^{\delta}   \frac{Q^p(r)}{\sh^\frac{p}{2}(r)} \bigg| \phi_{0}(r) \int_r^1   \chi(s) \partial_s\left( \phi_{\infty}(s) \sh^{-\frac{1}{2}}(s) \right) \sh(s) \frac{\re(\varepsilon_1)}{Q} ds \bigg|^p  \sh(r) dr \\
& + \int_0^{\delta}   \frac{Q^p(r)}{\sh^\frac{p}{2}(r)}    \bigg| \phi_{0}(r) \int_{\frac{1}{2}}^{\infty}\Tilde{\chi}(r)  \phi_{\infty}(s) \frac{1}{\sh^{\frac{1}{2}}(s)}  \partial_s\left(\sh(s) \frac{\re(\varepsilon_1)}{Q} \right) ds \bigg|^p  \sh(r) dr \\
& =: \int_0^{\delta} I_{1,1}(r) dr + \int_0^{\delta} I_{1,2}(r) dr + \int_0^{\delta} I_{1,3}(r) dr + \int_0^{\delta} I_{1,4}(r) dr.
 \end{align*}
Using Lemma \eqref{Behav_phi_0_small_r}, we have $\phi_0(s) \sh^{-\frac{1}{2}}(s)=1 + 0(s^{2m+\frac{3}{2}}),$ which yields $\partial_s \left( \phi_0(s) \sh^{-\frac{1}{2}}(s)  \right)=O(r^{2m+\frac{1}{2}}).$ We estimate each of the terms $I_{1,j}$ separately.

\begin{align*}
\int_0^{\delta} I_{1,1}(r) dr  &\lesssim  \int_0^{\delta} \frac{r^{mp}}{r^{\frac{p}{2}}} r^{\frac{p}{2}} (\log(r))^p \left| \int_0^r  s^{2m+\frac{1}{2}} s  \frac{\re(\varepsilon_1)}{s^{m}} ds \right|^p r \, dr  \\
& \lesssim \int_0^{\delta} r^{mp+1 } (\log(r))^p \left| \int_0^r  s^{m+\frac{3}{2}-\frac{1}{p}}   | \varepsilon_1 | s^\frac{1}{p} ds \right|^p  dr \\
& \lesssim \int_0^{\delta}  r^{mp+1 } (\log(r))^p \bigg( \int_0^r | \varepsilon_1|^p \, s \, ds  \bigg) \left( \int_0^r s^{(m+\frac{3}{2}-\frac{1}{p})p^{\prime}} ds \right)^{\frac{p}{p^{\prime}}} \\
& \lesssim \left\|  \varepsilon_1 \right\|_{L^p( r \leq \delta)}^p  \int_0^{\delta}  r^{mp+1 } (\log(r))^p r^{(m+\frac{3}{2}-\frac{1}{p})p+\frac{p}{p^{\prime}}} dr \\
& \lesssim \left\|  \varepsilon_1 \right\|_{L^p( r \leq \delta)}^p  
\end{align*}

\begin{align*}
 \int_0^{\delta} I_{1,2}(r) dr:&=   \int_0^{\delta}   \frac{Q^p(r)}{\sh^\frac{p}{2}(r)} 
 \bigg|\phi_{0}(r) \int_r^1  \partial_s\left( \chi(s)\right)  \phi_{\infty}(s) \sh^{-\frac{1}{2}}(s)  \sh(s) \frac{\re(\varepsilon_1)}{Q} ds \bigg|^p \sh(r) dr \\
 & \lesssim   \int_0^{\delta}   \frac{r^{mp}}{r^{\frac{p}{2}}} r^{\frac{p}{2}} \left| \int_r^1  \partial_s\left( \chi(s)\right)  s^{\frac{1}{2}} \log(s) s^{\frac{1}{2}} \frac{\re(\varepsilon_1)}{s^m} ds \right|^p r\, dr \\
 & \lesssim  \int_0^{\delta} r^{mp+1} \left| \int_r^1 \partial_s\left( \chi(s)\right)
 s^{-m+1- \frac{1}{p}} \log(s) | \varepsilon_1 | s^{\frac{1}{p}} ds 
 \right|^p dr \\
  & \lesssim  \int_0^{\delta} r^{mp+1} \left( \int_r^1| \varepsilon_1 |^p s ds \right) \left( \int_{\frac{1}{2}}^1   \partial_s\left( \chi(s)\right)^{p^{\prime}}
 s^{(-m+1- \frac{1}{p})p^{\prime} } \log(s)^{p^{\prime}} ds \right)^{\frac{p}{p^{\prime}}} \\
 & \lesssim \left\|  \varepsilon_1 \right\|_{L^p( r \leq 1)}^p  \int_0^{\delta} r^{mp+1} dr  \lesssim \left\|  \varepsilon_1 \right\|_{L^p( r \leq 1)}^p
\end{align*}

\begin{align*}
 \int_0^{\delta} I_{1,3}(r) dr &= \int_0^{\delta}   \frac{Q^p(r)}{\sh^\frac{p}{2}(r)} \bigg| \phi_{0}(r) \int_r^1   \chi(s) \partial_s\left( \phi_{\infty}(s) \sh^{-\frac{1}{2}}(s) \right) \sh(s) \frac{\re(\varepsilon_1)}{Q} ds \bigg|^p  \sh(r) dr \\
 & \lesssim  \int_0^{\delta}   \frac{r^{mp}}{r^{\frac{p}{2}}} r^{\frac{p}{2}} \bigg| \int_r^1   \chi(s) \partial_s\left( s^{\frac{1}{2}} \log(s) s^{-\frac{1}{2}}(s) \right) s \frac{\re(\varepsilon_1)}{s^m} ds \bigg|^p   r dr \\   
 & \lesssim  \int_0^{\delta}  r^{mp+1} \bigg| \int_r^1 \chi(s) s^{-m-\frac{1}{p}}  | \varepsilon_1 | s^{\frac{1}{p}} ds \bigg|^2 dr \\
  & \lesssim   \int_0^{\delta}  r^{mp+1} \left(\int_r^1 | \varepsilon_1 |^p s ds  \right)  \left(\int_r^1  \chi(s) s^{(-m-\frac{1}{p})p^{\prime}} ds \right)^{\frac{p}{p^{\prime}}} \\
  & \lesssim \left\|  \varepsilon_1 \right\|_{L^p( r \leq 1)}^p \int_0^{\delta}  r^{mp+1}  \; r^{(-m-\frac{1}{p})p+\frac{p}{p^{\prime}}} dr \\
   & \lesssim \left\|  \varepsilon_1 \right\|_{L^p( r \leq 1)}^p \int_0^{\delta}  r^{mp+1}  \; r^{-mp-1+\frac{p}{p^{\prime}}} dr \\ 
  & \lesssim \left\|  \varepsilon_1 \right\|_{L^p( r \leq 1)}^p \int_0^{\delta}  r^{\frac{p}{p^{\prime}}} dr  \lesssim \left\|  \varepsilon_1 \right\|_{L^p( r \leq 1)}^p
\end{align*}

\begin{align*}
 \int_0^{\delta} I_{1,4}(r) dr :&= \int_0^{\delta}   \frac{Q^p(r)}{\sh^\frac{p}{2}(r)}    \bigg| \phi_{0}(r) \int_{\frac{1}{2}}^{\infty}\Tilde{\chi}(s)  \phi_{\infty}(s) \frac{1}{\sh^{\frac{1}{2}}(s)}  \partial_s\left(\sh(s) \frac{\re(\varepsilon_1)}{Q} \right) ds \bigg|^p  \sh(r) dr \\
 & \lesssim  \int_0^{\delta}   \frac{Q^p(r)}{\sh^\frac{p}{2}(r)}    \bigg| \phi_{0}(r)  \left[ \Tilde{\chi}(s)  \phi_{\infty}(s)  \sh^{\frac{1}{2}}(s) \frac{\re(\varepsilon_1)}{Q} \right]_{s=\frac{1}{2}}^{s=\infty}  \\
 &\qquad \qquad \qquad - \phi_{0}(r) \int_{\frac{1}{2}}^{\infty} \partial_s\left( \Tilde{\chi}(s)  \phi_{\infty}(s) \frac{1}{\sh^{\frac{1}{2}}(s)}  \right)\sh(s) \frac{\re(\varepsilon_1)}{Q}  ds \bigg|^p  \sh(r) dr 
\end{align*}
Using the fact that $\Tilde{\chi}(r) = 0$ for $r\leq \frac{1}{2}$ and  $\phi_{\infty} (s) \sim   e^{-\frac{\sqrt{5}}{2} s}$ for large $r$, we obtain
$$ \left[ \Tilde{\chi}(s)  \phi_{\infty}(s)  \sh^{\frac{1}{2}}(s) \frac{\re(\varepsilon_1)}{Q} \right]_{s=\frac{1}{2}}^{s=\infty} =0.$$
Hence,
\begin{align*}
 \int_0^{\delta} I_{1,4}(r) dr   & \lesssim  \int_0^{\delta} \frac{Q^p(r)}{\sh^\frac{p}{2}(r)}   \bigg| \phi_{0}(r)   \int_{\frac{1}{2}}^{\infty} \partial_s\left( \Tilde{\chi}(s)  \phi_{\infty}(s) \frac{1}{\sh^{\frac{1}{2}}(s)}  \right)\sh(s) \frac{\re(\varepsilon_1)}{Q}  ds \bigg|^p  \sh(r) dr  \\
&\lesssim  \int_0^{\delta} \frac{Q^p(r)}{\sh^\frac{p}{2}(r)}   \bigg| \phi_{0}(r)   \int_{\frac{1}{2}}^{1} \partial_s\left( \Tilde{\chi}(s) \right)   \phi_{\infty}(s)  \sh^{\frac{1}{2}}(s) \frac{\re(\varepsilon_1)}{Q}  ds \bigg|^p  \sh(r) dr  \\
& +  \int_0^{\delta} \frac{Q^p(r)}{\sh^\frac{p}{2}(r)}   \bigg| \phi_{0}(r)   \int_{\frac{1}{2}}^{\infty} \Tilde{\chi}(s)  \partial_s\left( \phi_{\infty}(s) \frac{1}{\sh^{\frac{1}{2}}(s)}  \right)\sh(s) \frac{\re(\varepsilon_1)}{Q}  ds \bigg|^p  \sh(r) dr \\
  & \lesssim \int_0^{\delta}   \frac{r^{mp}}{r^{\frac{p}{2}}} r^{\frac{p}{2}} \bigg| \int_{\frac{1}{2}}^{1} s^{\frac{1}{2}} \log(s) s^{\frac{1}{2}} \frac{\re(\varepsilon_1)}{s^{m}}  ds \bigg|^p  r \, dr
 \\
 & + \int_0^{\delta}   \frac{r^{mp}}{r^{\frac{p}{2}}} r^{\frac{p}{2}} \bigg| \int_{\frac{1}{2}}^{\infty}  e^{-\frac{\sqrt{5}}{2} s } e^{-\frac{1}{2} s } e^{ s } \re(\varepsilon_1)  ds \bigg|^p  r \, dr \\
 & \lesssim \int_0^{\delta}   r^{mp+1} \bigg| \int_{\frac{1}{2}}^{1} s^{-m+1-\frac{1}{p}} \log(s)  \, \re(\varepsilon_1) s^{\frac{1}{p}}  ds \bigg|^p   dr
 \\
 & + \int_0^{\delta}   r^{mp+1} \bigg| \int_{\frac{1}{2}}^{\infty}  e^{(-\frac{\sqrt{5}}{2} +\frac{1}{2}-\frac{1}{p})s }   \re(\varepsilon_1) e^{\frac{1}{p}s} ds \bigg|^p   dr \\
 & \lesssim   \int_0^{\delta}   r^{mp+1} \left( \int_{\frac{1}{2}}^1 | \varepsilon_1|^p s ds \right) \left( \int_{\frac{1}{2}}^1  s^{(-m+1-\frac{1}{p})p^{\prime}} (\log(s))^{p^{\prime}}  ds \right)^{\frac{p}{p^{\prime}}} dr \\
 & + \int_0^{\delta}   r^{mp+1} \left( \int_{\frac{1}{2}}^{\infty} | \varepsilon_1|^p e^s ds \right)  \left(\int_{\frac{1}{2}}^{\infty}   e^{(-\frac{\sqrt{5}}{2} +\frac{1}{2}-\frac{1}{p})p^{\prime} s }   ds \right)^{\frac{p}{p^{\prime}}} dr \\
  & \lesssim \left\|  \varepsilon_1 \right\|_{L^p( r \leq 1)}^p +     \left\|  \varepsilon_1 \right\|_{L^p( r \geq 1)}^p 
\end{align*}
Then \begin{align*}
\int_0^{\delta}  I_1(r) \sh(r) dr &
\lesssim \left\|  \varepsilon_1 \right\|_{L^p}^p    
\end{align*}

Next we estimate $I_2(r),$
\begin{align*}
 \int_0^{\delta}  I_2(r) \sh(r) dr &\lesssim \int_0^{\delta}     \frac{Q^p(r)}{\sh^\frac{p}{2}(r)}  \bigg(  \bigg| \phi_{\infty}(r) \int_0^r \phi_0(s) \frac{1}{\sh^{\frac{1}{2}}(s)}  \partial_s \left( \sh(s) \frac{a_{\theta}}{\sh(s)} \frac{\re(\varepsilon)}{Q} \right)   ds  \\
& \qquad \qquad \qquad + \phi_{0}(r) \int_r^{\infty} \phi_{\infty}(s) \frac{1}{\sh^{\frac{1}{2}}(s)} \partial_s \left( \sh(s) \frac{a_{\theta}}{\sh(s)} \frac{\re(\varepsilon)}{Q} \right) ds \bigg|^p  \bigg)   \sh(r) dr 
\end{align*}
Similarly to the estimate obtained above for $I_1(r),$ we have 
 \begin{align*}
\int_0^{\delta}  I_2(r) \sh(r) dr &
\lesssim \left\|\frac{a_{\theta}}{\sh(r)} \re(\varepsilon)  \right\|_{L^p}^p     \lesssim  \left\| \frac{a_{\theta}}{\sh(r)} \right\|_{L^\infty}^p   \left\|  \varepsilon \right\|_{L^p}^p   \\
\end{align*}

Next we estimate $I_3(r),$ using the asymptotic behavior \eqref{asymp_phi0} for $\phi_0$  and \eqref{asymp_phi-infty} for $\phi_{\infty}.$ 

\begin{align*}
 \int_0^{\delta}  I_3(r) \sh(r) dr   
   &\lesssim   \int_0^{\delta}  \frac{Q^p(r)}{\sh^\frac{p}{2}(r)}  \bigg(\left| \phi_{\infty}(r)  \int_0^r \phi_0(s) \sh^{\frac{1}{2}}(s) |\varepsilon|^2  ds \right|^p
  \\
  &  \quad \qquad \qquad \qquad+ \frac{Q^p(r)}{\sh^\frac{p}{2}(r)} \left| \phi_{0}(r) \int_r^{1} \phi_{\infty}(s)  \sh^{\frac{1}{2}}(s) |\varepsilon|^2 ds \right|^p \bigg)  \sh(r) dr  \\
&   + \int_0^{\delta}  \frac{Q^p(r)}{\sh^\frac{p}{2}(r)}  \left| \phi_{0}(r) \int_1^{\infty} \phi_{\infty}(s)  \sh^{\frac{1}{2}}(s) |\varepsilon|^2 ds \right|^p   \sh(r) dr\\
  &\lesssim   \int_0^{\delta}   \frac{r^{mp+1}}{r^{\frac{p}{2}}}  r^{\frac{p}{2}} (\log(r))^{p} \left\|\varepsilon \right\|_{L^\infty}^p \left| \int_0^r s^{1-\frac{1}{p}}  |\varepsilon| s^{\frac{1}{p}} ds\right|^p \\
  & +\frac{r^{mp+1}}{r^{\frac{p}{2}}} r^{\frac{p}{2}} \left\|\varepsilon \right\|_{L^\infty}^p \left| \int_r^1 s^{1-\frac{1}{p}} \log(s)  |\varepsilon| s^{\frac{1}{p}} ds\right|^p dr \\
  & + \int_0^{\delta} \frac{r^{mp+1}}{r^{\frac{p}{2}}}  r^{\frac{p}{2}} \left\|\varepsilon \right\|_{L^\infty}^p  \left| \int_1^{\infty} e^{\left(-\frac{\sqrt{5}}{2} + \frac{1}{2}-\frac{1}{p} \right)s}  |\varepsilon| e^{\frac{1}{p}s} ds \right|^p \\
 &\lesssim  \left\|\varepsilon \right\|_{L^\infty}^p  \int_0^{\delta} r^{mp+1} (\log(r))^{p} \left( \int_0^r   |\varepsilon|^p s ds \right) \left(  \int_0^r s^{(1-\frac{1}{p})p^{\prime}} ds  \right)^{\frac{p}{p^{\prime}}} dr \\
 &+  \left\|\varepsilon \right\|_{L^\infty}^p  \int_0^{\delta} r^{mp+1}  \left( \int_0^r   |\varepsilon|^p s \, ds \right) \left( \int_r^1 s^{(1-\frac{1}{p})p^{\prime}} (\log(s))^{p^{\prime}}  ds \right) dr \\
 &+  \left\|\varepsilon \right\|_{L^\infty}^p  \int_0^{\delta} r^{mp+1} 
 \left(\int_1^{\infty}  |\varepsilon| e^{s} ds   \right) \left( \int_1^{\infty} e^{\left(-\frac{\sqrt{5}}{2} + \frac{1}{2}-\frac{1}{p} \right)p^{\prime} s}    ds \right)^{\frac{p}{p^{\prime}}} \\
&\lesssim  \left\|\varepsilon \right\|_{L^\infty}^p  \left\|\varepsilon \right\|_{L^p}^p 
\end{align*}

\textbf{Step 2: Estimate $I_j$ for r large:}
We first estimate $I_1$ for $r \geq \delta,$ for some $\delta >0,$ using the asymptotic \eqref{asymp_phi-infty} for $\phi_{\infty},$ Lemma \ref{Behav_phi_0_small_r}, H\"older's inequality with Schur's Test. 
\begin{align*}
 \int_{\delta}^{\infty}  I_1(r) \sh(r) dr & \lesssim    \int_{\delta}^{\infty}  \frac{Q^p(r)}{\sh^\frac{p}{2}(r)} \bigg(  \bigg| \phi_{\infty}(r) \int_0^r \phi_0(s) F_1(\varepsilon(s), \varepsilon_1(s)) ds \\
 &  \quad \qquad \qquad \qquad+
  \phi_{0}(r) \int_r^{\infty} \phi_{\infty}(s) F_1(\varepsilon(s), \varepsilon_1(s)) ds \bigg|^p  \bigg) \sh(r) dr \\
    & \lesssim    \int_{\delta}^{\infty}    \frac{Q^p(r)}{\sh^\frac{p}{2}(r)}  \bigg(  \bigg| \phi_{\infty}(r) \int_0^r \phi_0(s) \frac{1}{\sh^{\frac{1}{2}}(s)}  \partial_s\left(\sh(s) \frac{\re(\varepsilon_1)}{Q} \right) ds \\
& \qquad \qquad \qquad \qquad  + \phi_{0}(r) \int_{r}^{\infty}  \phi_{\infty}(s) \frac{1}{\sh^{\frac{1}{2}}(s)}  \partial_s\left(\sh(s) \frac{\re(\varepsilon_1)}{Q} \right) ds \bigg|^p  \bigg) \sh(r) dr    \\
& \lesssim \int_{\delta}^{\infty}   \frac{Q^p(r)}{\sh^\frac{p}{2}(r)}  \bigg(  \bigg| \phi_{\infty}(r) \left[  \phi_{0}(s)  \sh^{\frac{1}{2}}(s) \frac{\re(\varepsilon_1)}{Q} \right]_{s=0}^{s=r} \\
&  \quad \qquad \qquad \qquad - \phi_{\infty}(r)  \int_0^r\partial_s\left( \phi_0(s) \frac{1}{\sh^{\frac{1}{2}}(s)} \right) \sh(s) \frac{\re(\varepsilon_1)}{Q}  ds \\
&    \quad \qquad \qquad \qquad  + \phi_{0}(r)\left[  \phi_{\infty}(s)  \sh^{\frac{1}{2}}(s) \frac{\re(\varepsilon_1)}{Q} \right]_{s=r}^{s=\infty} \\
&  \quad \qquad \qquad \qquad-\phi_{0}(r)  \int_{r}^{\infty} \partial_s\left( \phi_{\infty}(s) \frac{1}{\sh^{\frac{1}{2}}(s)} \right) \sh(s) \frac{\re(\varepsilon_1)}{Q}  ds \bigg|^p  \bigg) \sh(r) dr    \\
  & \lesssim    \int_{\delta}^{\infty}    \frac{Q^p(r)}{\sh^\frac{p}{2}(r)}   \bigg| \phi_{\infty}(r) \int_0^1 \partial_s\left( \phi_0(s)   \frac{1}{\sh^{\frac{1}{2}}(s)} \right)  \sh(s) \frac{\re(\varepsilon_1)}{Q} ds  \bigg|^p \sh(r) dr \\
&  + \int_{\delta}^{\infty}   \frac{Q^p(r)}{\sh^\frac{p}{2}(r)}  \bigg( \bigg| \phi_{\infty}(r) \int_1^{r} \partial_s\left( \phi_{0}(s) \frac{1}{\sh^{\frac{1}{2}}(s)}\right)  \sh(s) \frac{\re(\varepsilon_1)}{Q} ds \\
&\qquad \qquad \qquad \quad + \phi_{0}(r) \int_{r}^{\infty} \partial_s\left( \phi_{\infty}(s) \frac{1}{\sh^{\frac{1}{2}}(s)} \right)  \sh(s) \frac{\re(\varepsilon_1)}{Q} ds \bigg|^p  \bigg) \sh(r) dr     \\
&  \lesssim  \int_{\delta}^{\infty}I_{1,5} dr  + \int_{\delta}^{\infty}I_{1,6} dr  
\end{align*}

Using Lemma \eqref{Behav_phi_0_small_r}, we have $\phi_0(s) \sh^{-\frac{1}{2}}(s)=1 + 0(s^{2m+\frac{3}{2}})$, which yields $\partial_s \left( \phi_0(s) \sh^{-\frac{1}{2}}(s)  \right)=O(r^{2m+\frac{1}{2}}).$ Thus we have
\begin{align*}
\int_{\delta}^{\infty}I_{1,5} dr &\lesssim  \int_{\delta}^{\infty}    \frac{Q^p(r)}{\sh^\frac{p}{2}(r)}   \bigg| \phi_{\infty}(r) \int_0^1 \partial_s\left( \phi_0(s)   \frac{1}{\sh^{\frac{1}{2}}(s)} \right)  \sh(s) \frac{\re(\varepsilon_1)}{Q} ds  \bigg|^p \sh(r) dr   \\ 
& \lesssim  \int_{\delta}^{\infty} e^{(-\frac{1}{2}p-\frac{\sqrt{5}}{2}p)  r}  \bigg| \int_0^{1}  s^{2m+\frac{1}{2}} s  \frac{\re(\varepsilon_1)}{s^m} ds \bigg|^p e^r dr \\
& \lesssim  \int_{\delta}^{\infty} e^{(-\frac{1}{2}p-\frac{\sqrt{5}}{2}p)  r}  \bigg| \int_0^{1}  s^{m+\frac{3}{2}-\frac{1}{p}}  |\re(\varepsilon_1) | s^{\frac{1}{p}} ds \bigg|^p e^r dr \\
& \lesssim  \int_{\delta}^{\infty} e^{(-\frac{1}{2}p-\frac{\sqrt{5}}{2}p+1)  r}  \left( \int_0^{1} | \varepsilon_1 |^p s \, ds  \right)
 \left( \int_0^{1} s^{(m+\frac{3}{2}-\frac{1}{p})p^{\prime}} ds  \right)^{\frac{p}{p^{\prime}}} dr \\
 & \lesssim  \left\| \varepsilon_1 \right\|_{L^p}^p   \int_{\delta}^{\infty}  e^{(\frac{1}{2}p-\frac{\sqrt{5}}{2}p)  r} dr \\
  & \lesssim  \left\| \varepsilon_1 \right\|_{L^p}^p  
\end{align*}

\begin{align*}
\int_{\delta}^{\infty}I_{1,6} dr  &\lesssim     \int_{\delta}^{\infty}   \frac{Q^p(r)}{\sh^\frac{p}{2}(r)}  \bigg( \bigg| \phi_{\infty}(r) \int_1^{r} \partial_s\left( \phi_{0}(s) \frac{1}{\sh^{\frac{1}{2}}(s)}\right)  \sh(s) \frac{\re(\varepsilon_1)}{Q} ds \\
&\qquad \qquad \qquad \quad + \phi_{0}(r) \int_{r}^{\infty} \partial_s\left( \phi_{\infty}(s) \frac{1}{\sh^{\frac{1}{2}}(s)} \right)  \sh(s) \frac{\re(\varepsilon_1)}{Q} ds \bigg|^p  \bigg) \sh(r) dr  \\
&\lesssim    \int_{\delta}^{\infty} e^{-\frac{p}{2}   r}  e^r \bigg| e^{-\frac{\sqrt{5}}{2}   r}\int_1^r \partial_s e^{(\frac{\sqrt{5}}{2}-\frac{1}{2})s} e^{s} \re(\varepsilon_1) ds   
+    e^{\frac{\sqrt{5}}{2}   r} \int_r^{\infty} \partial_s e^{(-\frac{\sqrt{5}}{2}-\frac{1}{2})s} e^{s} \re(\varepsilon_1) ds  \bigg|^p dr \\
& \lesssim   \int_{\delta}^{\infty} \bigg|   \int_1^r e^{(-\frac{1}{2}-\frac{\sqrt{5}}{2}+\frac{1}{p})   r} e^{(\frac{\sqrt{5}}{2}+\frac{1}{2}-\frac{1}{p})s}  e^{-\frac{s}{p}}|\varepsilon_1 | e^{\frac{s}{p}} ds  
+   \int_r^{\infty} e^{(-\frac{1}{2}+\frac{\sqrt{5}}{2}+\frac{1}{p})   r}   e^{(-\frac{\sqrt{5}}{2}+\frac{1}{2}-\frac{1}{p} )s}  |\varepsilon_1| e^{\frac{1}{p}s}ds \,   \bigg|^p dr \\
&\lesssim   \int_{\delta}^{\infty} \bigg| \int_1^{\infty} K(s,r)   |\varepsilon_1| e^{\frac{1}{p}s}ds \bigg|^p dr 
\end{align*}

where 
\begin{align*}
  K(s,r)&=   \mathbb{1}_{\{1 \leq s \leq r\}} e^{(-\frac{\sqrt{5}}{2}-\frac{1}{2}+\frac{1}{p})(r-s)}+ \mathbb{1}_{\{\delta \leq r \leq s\}}e^{(-\frac{\sqrt{5}}{2}+\frac{1}{2}-\frac{1}{p}) (s-r)}
\end{align*}

Note that $\left\| K(s,r) \right\|_{L^1(ds)} $ bounded uniformly in $r$ and $\left\| K(s,r) \right\|_{L^1(dr)} $ bounded uniformly in $s.$ Indeed, 

\begin{align*}
 \int_1^{\infty} | K(s,r) | ds&=  \int_1^{\infty} \mathbb{1}_{\{1 \leq s \leq r\}} e^{(-\frac{1}{2}-\frac{\sqrt{5}}{2}+\frac{1}{p})(r-s)}+ \mathbb{1}_{\{\delta \leq r \leq s\}}e^{(-\frac{1}{2}-\frac{\sqrt{5}}{2}+\frac{1}{p}) (s-r)} ds \\
 &= \int_1^r e^{(\frac{1}{2}+\frac{\sqrt{5}}{2}-\frac{1}{p})s} ds \, e^{(-\frac{1}{2}-\frac{\sqrt{5}}{2}+\frac{1}{p})r} + 
 \int_{r}^{\infty} e^{(-\frac{1}{2}-\frac{\sqrt{5}}{2}+\frac{1}{p}) s} ds\,  e^{(\frac{1}{2}+\frac{\sqrt{5}}{2}-\frac{1}{p}) r} \\
 &\lesssim 1
\end{align*}
\begin{align*}
 \int_1^{\infty} | K(s,r)  | dr&=\int_1^{\infty}  \mathbb{1}_{\{1 \leq s \leq r\}} e^{(-\frac{1}{2}-\frac{\sqrt{5}}{2}+\frac{1}{p})(r-s)}+ \mathbb{1}_{\{\delta \leq r \leq s\}}e^{(-\frac{1}{2}-\frac{\sqrt{5}}{2}+\frac{1}{p}) (s-r)} dr \\
 &= \int_1^r e^{(-\frac{1}{2}-\frac{\sqrt{5}}{2}+\frac{1}{p})r} dr \, e^{(\frac{1}{2}+\frac{\sqrt{5}}{2}-\frac{1}{p})s} + 
 \int_{r}^{\infty} e^{(\frac{1}{2}+\frac{\sqrt{5}}{2}-\frac{1}{p}) r} ds\,  e^{(-\frac{1}{2}-\frac{\sqrt{5}}{2}+\frac{1}{p}) s} \\
 &\lesssim 1    
\end{align*}
Then by Schur's Test, we have 

\begin{align*}
\int_{\delta}^{\infty} \bigg| \int_1^{\infty} K(s,r)   |\varepsilon_1| e^{\frac{1}{p}s}ds \bigg|^p dr \lesssim   \int_{\delta}^{\infty} 
|\varepsilon_1|^p e^{s}ds \lesssim \left\| \varepsilon_1 \right\|_{L^p}^p
\end{align*}
which yields, 
\begin{align*}
\int_{\delta}^{\infty}I_{1,6} dr \lesssim \left\| \varepsilon_1 \right\|_{L^p}^p 
\end{align*}
and \begin{align*}
     \int_{\delta}^{\infty}  I_1(r) \sh(r) dr & \lesssim \left\| \varepsilon_1 \right\|_{L^p}^p 
\end{align*}
Similarly to the estimate obtained above for $I_1(r),$ we have 
 \begin{align*}
\int_{\delta}^{\infty}  I_2(r) \sh(r) dr &
\lesssim \left\|\frac{a_{\theta}}{\sh(r)} \re(\varepsilon)  \right\|_{L^p}^p     \lesssim  \left\| \frac{a_{\theta}}{\sh(r)} \right\|_{L^\infty}^p   \left\|  \varepsilon \right\|_{L^p}^p  
\end{align*}

Next, we estimate $I_3:$ 
\begin{align*}
\int_{\delta}^{\infty}  I_3(r) \sh(r) dr   
   &\lesssim  \int_{\delta}^{\infty} \frac{Q^p(r)}{\sh^\frac{p}{2}(r)}  \bigg(\left| \phi_{\infty}(r)  \int_0^1 \phi_0(s) \sh^{\frac{1}{2}}(s) |\varepsilon|^2  ds \right|^p  \\
&    \quad \qquad \qquad \qquad   + \frac{Q^p(r)}{\sh^\frac{p}{2}(r)} \left| \phi_{\infty}(r) \int_1^{r} \phi_{0}(s)  \sh^{\frac{1}{2}}(s) |\varepsilon|^2 ds \right|^p \bigg)  \sh(r) dr  \\
&   +\int_{\delta}^{\infty} \frac{Q^p(r)}{\sh^\frac{p}{2}(r)}  \left| \phi_{0}(r) \int_r^{\infty} \phi_{\infty}(s)  \sh^{\frac{1}{2}}(s) |\varepsilon|^2 ds \right|^p   \sh(r) dr\\
&\lesssim \left\| \varepsilon \right\|_{L^{\infty}}^p \int_{\delta}^{\infty} e^{(-\frac{\sqrt{5}}{2} -\frac{1}{2}) pr} \bigg| \int_0^{1} s^{1-\frac{1}{p}} |\varepsilon| s^{\frac{1}{p}} ds  \bigg|^p e^r dr \\
&+ \left\| \varepsilon \right\|_{L^{\infty}}^p 
\int_{\delta}^{\infty}   \bigg| e^{(-\frac{\sqrt{5}}{2} -\frac{1}{2}+\frac{1}{p}) r} \int_1^r e^{(\frac{\sqrt{5}}{2} + \frac{1}{2})s} |\varepsilon| ds +      e^{(\frac{\sqrt{5}}{2}-\frac{1}{2}+\frac{1}{p})r } \int_r^{\infty} e^{(-\frac{\sqrt{5}}{2}+\frac{1}{2})s }   |\varepsilon| ds \bigg|^p   dr\\
&\lesssim \left\| \varepsilon \right\|_{L^{\infty}}^p \int_{\delta}^{\infty} e^{(-\frac{\sqrt{5}}{2} -\frac{1}{2}) pr} \left( \int_0^1 |\varepsilon|^p s \, ds \right) \left(\int_0^1 s^{(1-\frac{1}{p})p^{\prime}}  ds \right)^{\frac{p}{p^{\prime}}} \\
&+  \left\| \varepsilon \right\|_{L^{\infty}}^p 
\int_{\delta}^{\infty}   \bigg|  \int_1^r e^{(-\frac{\sqrt{5}}{2} -\frac{1}{2}+\frac{1}{p}) (r-s)}  |\varepsilon| e^{\frac{1}{p}s} ds +       \int_r^{\infty} e^{(-\frac{\sqrt{5}}{2} +\frac{1}{2}-\frac{1}{p}) (s-r)}   |\varepsilon| e^{\frac{1}{p}s} ds \bigg|^p   dr \\
&\lesssim \left\| \varepsilon \right\|_{L^{\infty}}^p  \left\|  \varepsilon \right\|_{L^{p}}^p +  \left\| \varepsilon \right\|_{L^{\infty}}^p 
\int_{\delta}^{\infty}   \bigg| \int_1^{\infty} K(s,r)   |\varepsilon| e^{\frac{1}{p}s} ds \bigg|^p   dr 
\end{align*}
where 
\begin{align*}
  K(s,r)&=   \mathbb{1}_{\{1 \leq s \leq r\}} e^{(-\frac{\sqrt{5}}{2}-\frac{1}{2}+\frac{1}{p})(r-s)}+ \mathbb{1}_{\{\delta \leq r \leq s\}}e^{(-\frac{\sqrt{5}}{2}+\frac{1}{2}-\frac{1}{p}) (s-r)}
\end{align*}
and  $\left\| K(s,r) \right\|_{L^1(ds)} $ bounded uniformly in $r$ and $\left\| K(s,r) \right\|_{L^1(dr)} $ bounded uniformly in $s.$ By Schur's test,  we have 
\begin{align*}
\int_{\delta}^{\infty} \bigg| \int_1^{\infty} K(s,r)   |\varepsilon|^p e^{\frac{1}{p}s}ds \bigg|^p dr \lesssim   \int_{\delta}^{\infty} 
|\varepsilon|^p  e^{s}ds \lesssim \left\| \varepsilon \right\|_{L^p}^p
\end{align*}
which yields, 

\begin{align*}
\int_{\delta}^{\infty}  I_3(r) \sh(r) dr   &\lesssim \left\| \varepsilon \right\|_{L^{\infty}}^p  \left\|  \varepsilon \right\|_{L^{p}}^p
\end{align*}
This concludes the proof of \eqref{estim-real-eps}. Next we estimate the $L^\infty$-norm of the real part of $\varepsilon$ in terms of $\varepsilon_1.$ Recall that, by \eqref{est-real-eps-p-norm} for $p=1$, we have

\begin{align*}
\left| \re(\varepsilon(r)) \right|
\lesssim |J_1(r)|+ |J_2(r)|+|J_3(r)|
\end{align*}
where 
 \begin{align*}
 J_1(r):&=     \frac{Q(r)}{\sh^\frac{1}{2}(r)}  \left| \phi_{\infty}(r) \int_0^r \phi_0(s) F_1(\varepsilon(s), \varepsilon_1(s)) ds + \phi_{0}(r) \int_r^{\infty} \phi_{\infty}(s) F_1(\varepsilon(s), \varepsilon_1(s)) ds \right| \\
J_2(r):&=   \frac{Q(r)}{\sh^\frac{1}{2}(r)}  \left| \phi_{\infty}(r) \int_0^r \phi_0(s) F_2(\varepsilon(s), \varepsilon_1(s)) ds +   \phi_{0}(r) \int_r^{\infty} \phi_{\infty}(s) F_2(\varepsilon(s), \varepsilon_1(s)) ds \right| \\
J_3(r):&=  \frac{Q(r)}{\sh^\frac{1}{2}(r)}  \left| \phi_{\infty}(r) \int_0^r \phi_0(s) F_3(\varepsilon(s), \varepsilon_1(s)) ds 
   +   \phi_{0}(r) \int_r^{\infty} \phi_{\infty}(s) F_3(\varepsilon(s), \varepsilon_1(s)) ds \right| 
\end{align*}
and $F_i(\varepsilon, \varepsilon_1)$ is defined by \eqref{def-F-1}, \eqref{def-F-2} and \eqref{def-F-3} for $i = 1, 2, 3$. \\
\begin{align*}
|J_1(r)|&=     \frac{Q(r)}{\sh^\frac{1}{2}(r)}  \bigg| \phi_{\infty}(r) \int_0^r \phi_0(s) \frac{1}{\sh^{\frac{1}{2}}(s)}  \partial_s \left(\sh(s) \frac{\re(\varepsilon_1)}{Q} \right) ds \\
&    \quad \qquad \qquad + \phi_{0}(r) \int_r^{\infty} \phi_{\infty}(s) \frac{1}{\sh^{\frac{1}{2}}(s)}  \partial_s \left( \sh(s) \frac{\re(\varepsilon_1)}{Q} \right) ds \bigg| \\ 
&  \lesssim \frac{Q(r)}{\sh^\frac{1}{2}(r)}  \bigg( \bigg| \phi_{\infty}(r) \left[  \phi_{0}(s)  \sh^{\frac{1}{2}}(s) \frac{\re(\varepsilon_1)}{Q} \right]_{s=0}^{s=r}
- \phi_{\infty}(r)  \int_0^r\partial_s\left( \phi_0(s) \frac{1}{\sh^{\frac{1}{2}}(s)} \right) \sh(s) \frac{\re(\varepsilon_1)}{Q}  ds
 \\
& \qquad \qquad \; +    \phi_{0}(r)  \left[  \phi_{\infty}(s)  \sh^{\frac{1}{2}}(s) \frac{\re(\varepsilon_1)}{Q} \right]_{s=r}^{s=\infty} -\phi_{0}(r)  \int_{r}^{\infty} \partial_s\left( \phi_{\infty}(s) \frac{1}{\sh^{\frac{1}{2}}(s)} \right) \sh(s) \frac{\re(\varepsilon_1)}{Q(s)}  ds \bigg| \bigg)
\end{align*}
Using Lemma \eqref{Behav_phi_0_small_r}, we have $\phi_0(s) \sh^{-\frac{1}{2}}(s)=1 + 0(s^{2m+\frac{3}{2}})$
, which yields $\partial_s \left( \phi_0(s) \sh^{-\frac{1}{2}}(s)  \right)=O(r^{2m+\frac{1}{2}})$

For small $r$, we have 
\begin{align*}
|J_1(r)| & \lesssim r^{m-\frac{1}{2}}  \bigg|  
\phi_{\infty}(r)  \int_0^r \partial_s\left( \phi_0(s) \frac{1}{\sh^{\frac{1}{2}}(s)} \right) \sh(s) \frac{\re(\varepsilon_1)}{Q(s)}  ds
\\
& \qquad \qquad + \phi_{0}(r)  \int_{r}^{1} \partial_s\left( \phi_{\infty}(s) \frac{1}{\sh^{\frac{1}{2}}(s)} \right) \sh(s) \frac{\re(\varepsilon_1)}{Q(s)}  ds \bigg|  \\
   & + r^{m-\frac{1}{2}} \bigg|  \phi_{0}(r) \int_1^{\infty} \partial_s\left( \phi_{\infty}(s) \frac{1}{\sh^{\frac{1}{2}}(s)} \right) \sh(s) \frac{\re(\varepsilon_1)}{Q(s)}  ds \bigg|  \\
& \lesssim r^{m} |\log(r)| \bigg| \int_0^r s^{m+\frac{3}{2}-\frac{1}{p_1}} \re(\varepsilon_1) s^{\frac{1}{p_1}} ds \bigg| +  r^{m} \bigg| \int_r^1 s^{-m-\frac{1}{p_1}} \re(\varepsilon_1) s^{\frac{1}{p_1}} ds \bigg| 
\\ &+  r^{m} \bigg| \int_1^{\infty} e^{(-\frac{\sqrt{5}}{2}+\frac{1}{2}-\frac{1}{p_1}) s }  \re(\varepsilon_1) e^{\frac{1}{p_1} s } ds  \\
& \lesssim r^{m} |\log(r)|  \left( \int_0^r |\varepsilon_1|^{p_1} s ds \right)^{\frac{1}{p_1}}   \left( \int_0^r s^{(m+\frac{3}{2} -\frac{1}{p_1})p_1^{\prime}} ds  \right)^{\frac{1}{p_1^{\prime}}}\\
&+ r^{m} \left( \int_r^1 | \varepsilon_1|^{p_1} e^s ds \right)^{\frac{1}{p_1}} \left( \int_r^1 s^{(-m -\frac{1}{p_1})p_1^{\prime}} ds  \right)^{\frac{1}{p_1^{\prime}}} \\
& + r^m \left( \int_1^{\infty}   |\varepsilon_1|^{p_1} e^s   ds \right)^{\frac{1}{p_1}}
\left( \int_r^{\infty}  e^{(-\frac{\sqrt{5}}{2}+\frac{1}{2}-\frac{1}{p_1}) p_1^{\prime}s }    \right)^{\frac{1}{p_1^{\prime}}} \\
& \lesssim \left\| \varepsilon_1 \right\|_{L^{p_1}}
\end{align*}

For large $r$, we have  
\begin{align*}
|J_1(r)| & \lesssim e^{-\frac{1}{2}r}  \bigg|  
\phi_{\infty}(r)  \int_0^1 \partial_s\left( \phi_0(s) \frac{1}{\sh^{\frac{1}{2}}(s)} \right) \sh(s) \frac{\re(\varepsilon_1)}{Q(s)}  ds \\
& \qquad  \qquad  + \phi_{\infty}(r)  \int_1^r \partial_s\left( \phi_0(s) \frac{1}{\sh^{\frac{1}{2}}(s)} \right) \sh(s) \frac{\re(\varepsilon_1)}{Q(s)}  ds \\
  & \qquad \qquad \quad + \phi_{0}(r)  \int_{r}^{\infty} \partial_s\left( \phi_{\infty}(s) \frac{1}{\sh^{\frac{1}{2}}(s)} \right) \sh(s) \frac{\re(\varepsilon_1)}{Q(s)}  ds \bigg|  \\
& \lesssim e^{-(\frac{1}{2}+\frac{\sqrt{5}}{2})r} \bigg| \int_0^1 s^{m+\frac{3}{2}-\frac{1}{p_1}} \re(\varepsilon_1) s^{\frac{1}{p_1}} ds \bigg| +  e^{-(\frac{1}{2}+\frac{\sqrt{5}}{2})r} \bigg| \int_1^r e^{(\frac{\sqrt{5}}{2}+\frac{1}{2}-\frac{1}{p_1})s} \re(\varepsilon_1) e^{\frac{1}{p_1}s} \, ds \bigg| \\
& + e^{-(\frac{1}{2}-\frac{\sqrt{5}}{2})r} \bigg| \int_r^{\infty} e^{(-\frac{\sqrt{5}}{2}+\frac{1}{2}-\frac{1}{p_1}) s }  \re(\varepsilon_1) e^{\frac{1}{p_1} s } ds  \\
& \lesssim e^{-(\frac{1}{2}+\frac{\sqrt{5}}{2})r} \left( \int_0^1 |\varepsilon_1|^{p_1} s ds \right)^{\frac{1}{p_1}}   \left( \int_0^1 s^{(m+\frac{3}{2} -\frac{1}{p_1})p_1^{\prime}} ds  \right)^{\frac{1}{p_1^{\prime}}} \\
&+ e^{-(\frac{1}{2}+\frac{\sqrt{5}}{2})r} \left( \int_1^r | \varepsilon_1|^{p_1} e^s ds \right)^{\frac{1}{p_1}} \left( \int_1^r  e^{(\frac{\sqrt{5}}{2}+\frac{1}{2}-\frac{1}{p_1})p_1^{\prime} s}  ds  \right)^{\frac{1}{p_1^{\prime}}} \\
& + e^{-(\frac{1}{2}-\frac{\sqrt{5}}{2})r} \left( \int_r^{\infty}   |\varepsilon_1|^{p_1} e^s   ds \right)^{\frac{1}{p_1}}
\left( \int_r^{\infty}  e^{(-\frac{\sqrt{5}}{2}+\frac{1}{2}-\frac{1}{p_1}) p_1^{\prime}s }    \right)^{\frac{1}{p_1^{\prime}}} \\
& \lesssim \left\| \varepsilon_1 \right\|_{L^{p_1}}
\end{align*}

Similarly, one can estimate $J_2,$
\begin{align*}
|J_2(r) |
& \lesssim \left\| \frac{a_{\theta}}{\sh(s)} \re(\varepsilon) \right\|_{L^{q_1}} \lesssim \left\| \frac{a_{\theta}}{\sh(s)} \right\|_{L^{\infty}}  \left\| \varepsilon \right\|_{L^{q_1}}
\end{align*}
Next, we estimate $J_3.$ 
\begin{align*}
J_3(r) = \frac{Q(r)}{\sh^\frac{1}{2}(r)}  \left| \phi_{\infty}(r) \int_0^r \phi_0(s)  \sh^{\frac{1}{2}}(s) |\varepsilon|^2  ds 
   +   \phi_{0}(r) \int_r^{\infty} \phi_{\infty}(s) \sh^{\frac{1}{2}}(s) |\varepsilon|^2  ds \right| 
\end{align*}

For $r$ small, we have  
\begin{align*}
 |J_3(r)| & \lesssim r^{m-\frac{1}{2}} \bigg| \phi_{\infty}(r) \int_0^r \phi_0(s) \sh^{\frac{1}{2}}(s) |\varepsilon|^2 ds 
 +   \phi_{0}(r) \int_r^{1} \phi_{\infty}(s)  \sh^{\frac{1}{2}}(s) |\varepsilon|^2 ds  \\
 & \qquad  \qquad  +   \phi_{0}(r) \int_1^{\infty} \phi_{\infty}(s)  \sh^{\frac{1}{2}}(s) |\varepsilon|^2 ds \bigg|      \\
& \lesssim  r^{m} |\log(r)|  \int_0^r s |\varepsilon|^2 ds + 
r^{m} \int_r^1 s \log(s) |\varepsilon|^2 ds +  r^{m}   \int_1^{\infty} e^{(-\frac{\sqrt{5}}{2}+\frac{1}{2})s}  |\varepsilon|^2 ds   \\
& \lesssim   \left\| \varepsilon \right\|_{L^{\infty}}^2
\end{align*}
For large $r$ we have 
\begin{align*}
  |J_3(r)|  & \lesssim e^{-\frac{1}{2} r } \bigg| \phi_{\infty}(r) \int_0^1 \phi_0(s) \sh^{\frac{1}{2}}(s) |\varepsilon|^2 ds +\phi_{\infty}(r) \int_1^r \phi_0(s) \sh^{\frac{1}{2}}(s) |\varepsilon|^2 ds \\
 & \qquad  \qquad  
   +   \phi_{0}(r) \int_r^{\infty} \phi_{\infty}(s)  \sh^{\frac{1}{2}}(s) |\varepsilon|^2 ds \bigg|      \\
& \lesssim e^{-(\frac{1}{2} + \frac{\sqrt{5}}{2}) r } \int_0^1 s |\varepsilon|^2 ds + e^{-(\frac{1}{2} + \frac{\sqrt{5}}{2}) r } \int_1^r e^{(\frac{\sqrt{5}}{2} + \frac{1}{2})r} |\varepsilon|^2 ds 
+ e^{-(\frac{1}{2}-\frac{\sqrt{5}}{2})r}   \int_r^{\infty} e^{(-\frac{\sqrt{5}}{2}+\frac{1}{2})s}  |\varepsilon|^2 ds   \\
& \lesssim   \left\| \varepsilon \right\|_{L^{\infty}}^2
\end{align*}
This concludes the proof of \eqref{real-eps-infty-norm} and Claim \ref{claim-real-eps-norm}.
\end{proof}

Next, we estimate the imaginary part of $\varepsilon$ in terms of $\varepsilon_1$
\begin{claim}
\label{Claim-img-estimate}
Let $2 \leq p,p_1,q_1 < \infty $
\begin{align}
\label{estim-imag-eps}
\left\| \im(\varepsilon) \right\|_{L^p} \lesssim \left\| \varepsilon_1 \right\|_{L^{p}} + \left\| \frac{a_\theta}{\sh(r)} \right\|_{L^{\infty}} \left\|\varepsilon \right\|_{L^{p}} 
\end{align}
\begin{align}
\label{imag-eps-infty-norm}
\left\| \im(\varepsilon) \right\|_{L^\infty} \lesssim \left\| \varepsilon_1 \right\|_{L^{p_1}} + \left\| \frac{a_\theta}{\sh(r)} \right\|_{L^{\infty}} \left\|\varepsilon \right\|_{L^{q_1}} 
\end{align}

\end{claim}

\begin{proof}
Recall that by \eqref{Im-eq_eps}, we have 
\begin{align*}
 \im(\varepsilon(r))=  Q(r) \int_r^{\infty}   \frac{ \im(\varepsilon_1(s))}{Q(s)}  ds - Q(r) \int_r^{\infty}   \frac{1}{Q(s)}  \frac{a_{\theta}(s)}{\sh(s)}  \im(\varepsilon(s)) ds        
\end{align*}

Let $0<\delta<1.$ Then using the asymptotics of $Q$ for small $r<\delta,$ 
we have 
 \begin{align*}
\int_0^{\delta} |\im(\varepsilon(r))|^p     \sh(r) dr & \lesssim 
\int_0^{\delta}  Q^p(r) \bigg|  \int_r^{\infty}   \frac{ \im(\varepsilon_1(s))}{Q(s)}  ds \bigg|^p \sh(r) dr  \\
& \; + \int_0^{\delta}  Q^p(r) \bigg| \int_r^{\infty}   \frac{1}{Q(s)}  \frac{a_{\theta}(s)}{\sh(s)}  \im(\varepsilon(s)) ds \bigg|^p \sh(r) dr \\
& \lesssim 
\int_0^{\delta}  r^{mp+1} \bigg|  \int_r^{1}   s^{-m-\frac{1}{p}} \im(\varepsilon_1(s)) s^{\frac{1}{p}} ds \bigg|^p  dr \\
&+ \int_0^{\delta}  r^{m p+1} \bigg| \int_1^{\infty} e^{-\frac{1}{p}s}  \im(\varepsilon_1(s))e^{\frac{1}{p}s}  ds \bigg|^p  dr \\
& + \left\|\frac{a_{\theta}}{\sh(r)} \right\|_{L^{\infty}}^p   \bigg(  \int_0^{\delta}  r^{mp+1} \bigg| \int_r^{1}   s^{-m-\frac{1}{p}}    \im(\varepsilon(s)) s^{\frac{1}{p}}ds \bigg|^p  dr  \\
&+ 
 \int_0^{\delta}  r^{mp+1}  \bigg| \int_1^{\infty}  e^{-\frac{1}{p}s}    \im(\varepsilon(s)) e^{\frac{1}{p}s} ds \bigg|^p  dr  \bigg) \\
 & \lesssim \left\|\varepsilon_1  \right\|_{L^p}^p \int_0^{\delta} r^{mp+1} \left( \int_r^1 s^{(-m-\frac{1}{p})p^{\prime}} ds \right)^{\frac{p}{p^{\prime}}} dr \\
 & + \left\|\varepsilon_1  \right\|_{L^p}^p \int_0^{\delta} r^{mp+1} \left( \int_1^{\infty} e^{-\frac{p^{\prime}}{p}s} ds \right)^{\frac{p}{p^{\prime}}} dr \\
& + \left\|\frac{a_{\theta}}{\sh(r)} \right\|_{L^{\infty}}^p   \left\|\varepsilon  \right\|_{L^p}^p \bigg(   \int_0^{\delta} r^{mp+1 } \left( \int_r^1 s^{(-m-\frac{1}{p})p^{\prime}} ds 
 \right)^{\frac{p}{p^{\prime}}} dr \bigg)  \\
&+   \left\|\frac{a_{\theta}}{\sh(r)} \right\|_{L^{\infty}}^p   \left\|\varepsilon  \right\|_{L^p}^p \bigg( \int_0^{\delta} r^{mp+1} \left( \int_1^{\infty} e^{-\frac{p^{\prime}}{p} s}   ds\right)^{\frac{p}{p^{\prime}}}  dr  \bigg) \\
 &\lesssim  \left\|\varepsilon_1  \right\|_{L^p}^p  + \left\|\frac{a_{\theta}}{\sh(r)} \right\|_{L^{\infty}}^p   \left\|\varepsilon  \right\|_{L^p}^p
 \end{align*}  

Similarly, using the asymptotics of $Q$ for large $r>\delta,$ for some $\delta >1,$ we have 

\begin{align*}
\int_{\delta}^{\infty} |\im(\varepsilon(r))|^p     \sh(r) dr & \lesssim 
\int_{\delta}^{\infty}  Q^p(r) \bigg|  \int_r^{\infty}   \frac{ \im(\varepsilon_1(s))}{Q(s)}  ds \bigg|^p \sh(r) dr  \\
&+ \int_{\delta}^{\infty} Q^p(r) \bigg| \int_r^{\infty}   \frac{1}{Q(s)}  \frac{a_{\theta}(s)}{\sh(s)}  \im(\varepsilon(s)) ds \bigg|^p \sh(r) dr \\
& \lesssim 
\int_{\delta}^{\infty}   \bigg|  \int_r^{\infty}   e^{-\frac{1}{p}s} \im(\varepsilon_1(s)) e^{\frac{1}{p}s} ds \bigg|^p e^r dr  \\& + \left\|\frac{a_{\theta}}{\sh(r)} \right\|_{L^{\infty}}^p  \int_{\delta}^{\infty} \bigg| \int_r^{\infty}   e^{-\frac{1}{p}s}  \im(\varepsilon(s)) e^{\frac{1}{p}s} ds \bigg|^p e^r dr \\
& \lesssim 
\int_{\delta}^{\infty}   \bigg|  \int_0 ^{\infty}   \mathbb{1}_{ \{r \leq s \}} e^{(r-s)\frac{1}{p}}  (\im(\varepsilon_1(s)) e^{\frac{1}{p}s}) ds \bigg|^p  dr  \\
& + \left\|\frac{a_{\theta}}{\sh(r)} \right\|_{L^{\infty}}^p  \int_{\delta}^{\infty} \bigg| \int_0^{\infty}   \mathbb{1}_{ \{r \leq s \}} e^{(r-s)\frac{1}{p}}  \im(\varepsilon(s)) e^{\frac{1}{p}s} ds \bigg|^p dr \\
& \lesssim 
\int_{\delta}^{\infty}   \bigg|  \int_0 ^{\infty}   K(s,r) (\im(\varepsilon_1(s)) e^{\frac{1}{p}s}) ds \bigg|^p  dr  \\
& + \left\|\frac{a_{\theta}}{\sh(r)} \right\|_{L^{\infty}}^p  \int_{\delta}^{\infty} \bigg| \int_0^{\infty}   K(s,r)  \im(\varepsilon(s)) e^{\frac{1}{p}s} ds \bigg|^p dr 
\end{align*}
where 
\begin{align*}
 K(s,r)=   \mathbb{1}_{ \{r \leq s \}} e^{(r-s)\frac{1}{p}}
\end{align*}
Note that $\left\| K(s,r) \right\|_{L^1(ds)} $ bounded uniformly in $r$ and $\left\| K(s,r) \right\|_{L^1(dr)} $ bounded uniformly in $s.$
Then by Schur's test, we have

\begin{align*}
\int_{\delta}^{\infty} \bigg| \int_0^{\infty} K(s,r)  (\im(\varepsilon_1(s)) e^{\frac{1}{p}s}) \bigg|^p dr \lesssim   \int_{\delta}^{\infty} 
|\varepsilon_1|^p e^{s}ds \lesssim \left\| \varepsilon_1 \right\|_{L^p}^p 
\end{align*}
and 
\begin{align*}
\int_{\delta}^{\infty} \bigg| \int_0^{\infty} K(s,r)  (\im(\varepsilon (s)) e^{\frac{1}{p}s}) \bigg|^p dr \lesssim   \int_{\delta}^{\infty} 
|\varepsilon_1|^p e^{s}ds \lesssim \left\| \varepsilon \right\|_{L^p}^p 
\end{align*}
which yields, 
 \begin{align*}
\int_{\delta}^{\infty} |\im(\varepsilon(r))|^p   \sh(r) dr \lesssim   \left\| \varepsilon_1 \right\|_{L^p}^p + \left\| \frac{a_\theta}{\sh(r)} \right\|_{L^{\infty}}^p  \left\|\varepsilon \right\|_{L^{p}}^p 
\end{align*}
This concludes the proof of \eqref{estim-imag-eps}. Next, we estimate the $L^{\infty}$-norm of the imaginary part of $\varepsilon$ in terms of $\varepsilon_1.$ By \eqref{Im-eq_eps}, we have
\begin{align*}
 \im(\varepsilon(r))=  Q(r) \int_r^{\infty}   \frac{ \im(\varepsilon_1(s))}{Q(s)}  ds - Q(r) \int_r^{\infty}   \frac{1}{Q(s)}  \frac{a_{\theta}(s)}{\sh(s)}  \im(\varepsilon(s)) ds        
\end{align*}

For $r$ small, we have
\begin{align*}
 |\im(\varepsilon(r))| &\lesssim \bigg| Q(r) \int_r^{\infty}   \frac{ \im(\varepsilon_1(s))}{Q(s)}  ds - Q(r) \int_r^{\infty}   \frac{1}{Q(s)}  \frac{a_{\theta}(s)}{\sh(s)}  \im(\varepsilon(s)) ds    \bigg| \\
 & \lesssim r^m \bigg| \int_r^{1}  s^{-m} \im(\varepsilon_1(s))  ds + 
  \int_{1}^{\infty}   \im(\varepsilon_1(s))  ds
 \bigg| \\
 &+ r^m \bigg| \int_r^1 s^{-m} \frac{a_{\theta}(s)}{\sh(s)} \im(\varepsilon(s)) ds + \int_1^{\infty}  \frac{a_{\theta}(s)}{\sh(s)} \im(\varepsilon(s)) ds \bigg| \\
 & \lesssim r^m \left( \int_r^1 |\varepsilon_1|^{p_1} s\, ds  \right)^{\frac{1}{p_1}} \left( \int_r^1 s^{(-m-\frac{1}{p_1})p_1^{\prime}} ds   \right)^{\frac{1}{p_1^{\prime}}} \\
 &+ r^m \left( \int_1^\infty |\varepsilon_1|^{p_1} e^s ds \right)^{\frac{1}{p_1}} \left( \int_1^{\infty} e^{-\frac{p_1^{\prime}}{p_1} s } ds\right)^{\frac{1}{p_1^{\prime}}}  \\
 & + r^m \left\|  \frac{a_{\theta}}{\sh(r)} \right\|_{L^\infty} \left( \int_r^1 |\varepsilon(s) |^{q_1} s \, ds \right)^{\frac{1}{q_1}} \left( \int_r^1 s^{(-m-\frac{1}{q_1})q_1^{\prime}} ds \right)^{\frac{1}{q_1^{\prime}}} \\
 &+ r^m \left\|  \frac{a_{\theta}}{\sh(r)} \right\|_{L^\infty} \left( \int_1^\infty |\varepsilon(s)|^{q_1} e^s ds \right)^{\frac{1}{q_1}} \left( \int_1^{\infty} e^{-\frac{q_1^{\prime}}{q_1} s } ds\right)^{\frac{1}{q_1^{\prime}}} \\
 & \lesssim \left\| \varepsilon_1 \right\|_{L^{p_1}} + \left\|  \frac{a_{\theta}}{\sh(r)} \right\|_{L^\infty} \left\| \varepsilon \right\|_{L^{q_1}}
\end{align*}
For $r$ large we have, 
\begin{align*}
 |\im(\varepsilon(r))| &\lesssim \bigg| Q(r) \int_r^{\infty}   \frac{ \im(\varepsilon_1(s))}{Q(s)}  ds - Q(r) \int_r^{\infty}   \frac{1}{Q(s)}  \frac{a_{\theta}(s)}{\sh(s)}  \im(\varepsilon(s)) ds    \bigg| \\
 & \lesssim \bigg| \int_r^{\infty} e^{-\frac{1}{p_1} s} \im(\varepsilon_1(s)) e^{\frac{1}{p_1} s} ds + \int_r^{\infty} \frac{a_{\theta}(s)}{\sh(s)} e^{-\frac{1}{p_1} s}  \im(\varepsilon(s))e^{\frac{1}{p_1} s}   ds \\
  & \lesssim   \left( \int_r^{\infty} |\varepsilon_1(s)|^{p_1} e^{ s} ds  \right)^{\frac{1}{p_1}} \left( \int_r^{\infty} e^{-\frac{p_1^{\prime}}{p_1} s}  ds\right)^{\frac{1}{p_1^{\prime}}} \\
  &+ \left\|  \frac{a_{\theta}}{\sh(r)} \right\|_{L^\infty} \left( \int_r^{\infty} |\varepsilon (s)|^{q_1} e^{ s} ds  \right)^{\frac{1}{q_1}} \left( \int_r^{\infty} e^{-\frac{q_1^{\prime}}{q_1} s} ds \right)^{\frac{1}{q_1^{\prime}}} \\
  & \lesssim \left\|\varepsilon_1 \right\|_{L^{p_1}} +\left\|  \frac{a_{\theta}}{\sh(r)} \right\|_{L^\infty} \left\|\varepsilon  \right\|_{L^{q_1}}
\end{align*}

This concludes the proof of Claim \ref{Claim-img-estimate}.
\end{proof}
 \end{proof}
\begin{lemma}
\label{claim-estim_a_theta/sinh}
Let $2\leq p \leq \infty,$ then we have
 \begin{align*}
  \left\| \frac{a_\theta}{\sh(r)} \right\|_{L^{p}} & \lesssim (1+\|\varepsilon\|_{L^\infty})(\left\| \varepsilon \right\|_{L^{\infty}}  + \left\| \varepsilon \right\|_{L^p}).
 \end{align*}  
\end{lemma}
\begin{proof}
Recall that by \eqref{equ-integral-a_theta/sinh}, we have 
\begin{align*}
\frac{a_{\theta}(r) }{\sh(r)} =  - \frac{1}{2\sh(r)} \int_0^r |\varepsilon|^2 \sh(s) ds  - \frac{1}{\sh(r)}\int_0^r Q \re(\varepsilon) \sh(s) ds 
\end{align*}

\begin{align*}
\left| \frac{a_{\theta}(r) }{\sh(r)} \right| & \lesssim \frac{1}{\sh(r) } \left\| \varepsilon \right\|_{L^{\infty}}^2   \int_0^r  \sh(s) ds + \frac{1}{\sh(r) } \left\| \varepsilon \right\|_{L^{\infty}}  \int_0^r  \sh(s) ds \\
& \lesssim  \left\| \varepsilon \right\|_{L^{\infty}}^2 + \left\| \varepsilon \right\|_{L^{\infty}}
\end{align*} 

Thus, we obtain 
\begin{align*}
    \left\| \frac{a_\theta}{\sh(r)} \right\|_{L^{\infty}} & \lesssim \left\| \varepsilon \right\|_{L^{\infty}} + \left\| \varepsilon \right\|_{L^{\infty}}^2.
\end{align*}
Next we estimate the $L^p-$norm for $2 \leq p <\infty.$ For small $r$ we have
\begin{align*}
 \int_0^{\delta}    \left| \frac{a_\theta}{\sh(r)} \right|^p \sh(r) dr \lesssim \left\| \frac{a_\theta}{\sh(r)}\right\|^p_{L^{\infty}} \left| 
\int_0^{\delta} \sh(r) dr \right|\lesssim \left\| \varepsilon \right\|_{L^{\infty}}^{p} +\left\| \varepsilon \right\|_{L^{\infty}}^{2p}.
\end{align*} 
Using \eqref{equ-integral-a_theta/sinh}, we obtain the following estimate for large $r$
\begin{align*}
 \int_{\delta}^{\infty}       \left| \frac{a_\theta}{\sh(r)} \right|^p \sh(r) dr & \lesssim  \int_{\delta}^{\infty}  \frac{1}{\sh^p(r)}     \left|  \int_0^1  \left( |\varepsilon|^2 + Q \re(\varepsilon) \right)  \sh(s) ds  \right|^p \sh(r) dr \\
 & +  \int_{\delta}^{\infty}  \frac{1}{\sh^p(r)} \left| \int_1^{r} \left( |\varepsilon|^2 + Q \re(\varepsilon) \right)  \sh(s) ds  \right|^p \sh(r) dr \\
 & \lesssim  \left\| \varepsilon \right\|_{L^{\infty}}^{2p}  \int_{\delta}^{\infty} e^{(1-p) r} \left| \int_0^1 \sh(s)  ds\right|^p  dr  + \left\| \varepsilon \right\|_{L^{\infty}}^{p}  \int_{\delta}^{\infty} e^{(1-p) r} \left| \int_0^1 \sh(s) s^{m}  ds\right|^p dr \\
 & + \left\| \varepsilon  \right\|_{L^{\infty}}^p \int_{\delta}^{\infty} e^{(1-p) r}  \left| \int_1^r e^{(1-\frac{1}{p})s} |\varepsilon | e^{\frac{1}{p}s}  ds\right|^p  dr + \int_{\delta}^{\infty} e^{(1-p) r}  \left| \int_1^r e^{(1-\frac{1}{p})s} |\varepsilon | e^{\frac{1}{p}s}  ds\right|^p  dr \\
 & \lesssim \left\| \varepsilon \right\|_{L^{\infty}}^{2p} + \left\| \varepsilon \right\|_{L^{\infty}}^{p} + \left( \left\| \varepsilon \right\|_{L^{\infty}}^{p} +1 \right)  \int_{\delta}^{\infty} \left| \int_1^r e^{(1-\frac{1}{p})(s-r)}  |\varepsilon | e^{\frac{1}{p}s} ds \right|^p dr \\
 & \lesssim \left\| \varepsilon \right\|_{L^{\infty}}^{2p} + \left\| \varepsilon  \right\|_{L^{\infty}}^p + \left(1+ \left\| \varepsilon \right\|_{L^{\infty}}^p \right) 
 \int_{\delta}^{\infty}   \left| \int_{1}^{\infty} K(s,r) | \varepsilon | e^{\frac{1}{p}s}  ds \right|^p   dr 
\end{align*}

where 
\begin{align*}
 K(s,r)=   \mathbb{1}_{ \{s \leq r \}} e^{-(r-s)(1-\frac{1}{p})}
\end{align*}
Note that $\left\| K(s,r) \right\|_{L^1(ds)} $ bounded uniformly in $r$ and $\left\| K(s,r) \right\|_{L^1(dr)} $ bounded uniformly in $s.$ 
Indeed, 
\begin{align*}
 \int_1^{\infty}   |K(s,r)| ds & \lesssim \int_1^r e^{s(1-\frac{1}{p})} ds
e^{-r(1-\frac{1}{p})} \lesssim 1 \\
 \int_1^{\infty}   |K(s,r)| dr & \lesssim \int_1^r e^{-r(1-\frac{1}{p})} dr  e^{s(1-\frac{1}{p})} \lesssim 1 
\end{align*}

Then by Schur's test, we have
\begin{align*}
    \int_{\delta}^{\infty}  \left| \int_{\delta}^{\infty}  K(s,r) | \varepsilon | e^{\frac{1}{p}s}  ds \right|^p   dr \lesssim   \int_{\delta}^{\infty} 
|\varepsilon_1|^p e^{s}ds \lesssim \left\| \varepsilon \right\|_{L^p}^p
\end{align*}
which yields, 
\begin{align*}
 \int_{\delta}^{\infty}       \left| \frac{a_\theta}{\sh(r)} \right|^p \sh(r) dr & \lesssim   \left\| \varepsilon \right\|_{L^{\infty}}^{2p} + \left\| \varepsilon  \right\|_{L^{\infty}}^p + \left(1+ \left\| \varepsilon \right\|_{L^{\infty}}^p \right)   \left\| \varepsilon \right\|_{L^p}^p \\
 & \lesssim \left\| \varepsilon \right\|_{L^{\infty}}^{p} + \left\| \varepsilon \right\|_{L^{\infty}}^{p} \left\| \varepsilon \right\|_{L^p}^p + \left\| \varepsilon \right\|_{L^p}^p+\left\| \varepsilon \right\|_{L^{\infty}}^{2p}.
\end{align*}
This concludes the proof of Lemma \ref{claim-estim_a_theta/sinh}.
\end{proof}

\begin{lemma} 
\label{lem:est-coth_a_theta/sinh-eps_1}
 Let $1< p,p_1,p_2 <\infty$ such that $\frac{1}{p_1}+\frac{1}{p_2}=\frac{1}{p}.$ Then, we have
\begin{align*}
 \left\|  \coth(r)  \frac{a_{\theta}}{\sh(r)}     \varepsilon_1 \right\|_{L^p} \lesssim (1+\|\varepsilon\|_{L^\infty})(\left\| \varepsilon \right\|_{L^{\infty}}  + \left\| \varepsilon \right\|_{L^{p_1}}) \left\|  \varepsilon_1  \right\|_{L^{p_2} }.
\end{align*}
 \end{lemma}
 \begin{proof}
 Using H\"older's inequality, we have 
\begin{align*}
\left\| \coth(r) \frac{a_{\theta}}{\sh(r)}     \varepsilon_1  \right\|_{L^p(r < \delta )} \lesssim \left\|  \coth(r) \frac{a_{\theta}(r)}{\sh(r)}    \right\|_{L^{p_1}(r < \delta )} \left\|    \varepsilon_1   \right\|_{L^{p_2}(r < \delta)}
\end{align*}
 
 Using \eqref{equ-integral-a_theta/sinh}, we obtain 
\begin{align*}
\left\|  \coth(r) \frac{a_{\theta}(r)}{\sh(r)}    \right\|_{L^{p_1}(r \leq \delta)}^{p_1} 
&\lesssim  \int_0^{\delta}  \frac{\coth^{p_1}(r)}{\sh^{p_1}(r)} \left| \left(  \int_0^r |\varepsilon(s)|^2 \sh(s) ds + \int_0^r Q(s) \re(\varepsilon(s)) \sh(s) ds   \right)  \right|^{p_1} \sh(r) dr \\
 &\lesssim \int_0^{\delta} \frac{1}{r^{2{p_1}-1}} \left|   \int_0^r |\varepsilon(s)|^2 \,s\, ds  \right|^{p_1} dr  
+ \int_0^{\delta} \frac{1}{r^{2{p_1}-1}} \left| \int_0^r s^m
\re(\varepsilon(s))\, s \, ds    \right|^{p_1}    dr \\
 & \lesssim  \left\| \varepsilon \right\|_{L^{\infty}}^{2{p_1}}  \int_0^{\delta} \frac{1}{r^{2{p_1}-1}} r^{2{p_1}} dr +   \left\| \varepsilon \right\|_{L^{\infty}}^{p_1} \int_0^{\delta} \frac{1}{r^{2{p_1}-1}}  r^{mp_1+2{p_1}} dr\\
 & \lesssim  (1+\|\varepsilon\|_{L^\infty}^{p_1})\left\| \varepsilon \right\|_{L^{\infty}}^{p_1}
\end{align*}

For large $r, $ we use H\"older's inequality along with the fact $\coth(r)$ is bounded, together with Lemma \ref{claim-estim_a_theta/sinh}. 
\begin{align*}
\left\| \coth(r) \frac{a_{\theta}}{\sh(r)}     \varepsilon_1  \right\|_{L^p(r \geq  1 )} &\lesssim     \left\|  \frac{a_{\theta}}{\sh(r)}       \right\|_{L^{p_1} }  \left\|  \varepsilon_1  \right\|_{L^{p_2} } \\
& \lesssim (1+\|\varepsilon\|_{L^\infty})(\left\| \varepsilon \right\|_{L^{\infty}}  + \left\| \varepsilon \right\|_{L^{p_1}}) \left\|  \varepsilon_1  \right\|_{L^{p_2} }.
\end{align*}

\end{proof}

\begin{lemma} \label{lem:A_0-Lp}
Let $2\leq p < \infty$
\begin{align*}
 \left\| A_0 \right\|_{L^p}   \lesssim (1+\left\| \varepsilon  \right\|_{L^\infty}) \left\| \varepsilon_1 \right\|_{L^p}.
\end{align*}
\end{lemma}
\begin{proof}
Recall that by \eqref{eq:A0toep1}
\begin{align*}
  A_0  = \frac{1}{2 } \int_r^{\infty} \re(\varepsilon_1) Q + \re(\varepsilon_1 \bar{\varepsilon}) \, ds
\end{align*}
For small $r,$ we have
\begin{align*}
 \int_0^{\delta}  |A_0|^p \sh(r) dr & \lesssim  \int_0^{\delta}    \left| \int_r^{1} \re(\varepsilon_1) s^{m} ds \right|^p r dr  + \int_0^{\delta} \left|  \int_1^{\infty} \re(\varepsilon_1) ds \right|^p r dr + \int_0^{\delta}    \left|  \int_r^{\infty} \re(\varepsilon_1 \bar{\varepsilon}) \, ds    \right|^p r dr \\   
 & \lesssim \int_0^{\delta}  \left( \int_r^1 |\varepsilon_1|^p s ds \right)   \left(\int_r^1 s^{(m-\frac{1}{p})p^{\prime}}   ds\right)^{\frac{p}{p^{\prime}}} r dr \\
 &+ \int_0^{\delta} \left( \int_1^{\infty} |\varepsilon_1|^p \sh(s) ds   \right) \left( \int_1^{\infty} (\sh(s))^{-\frac{p^{\prime}}{p}s} ds \right)^{\frac{p}{p^{\prime}}} r dr \\ 
 & + \left\| \varepsilon \right\|_{L^{\infty}}^p \int_0^{\delta} \left( \int_r^{\infty} |\varepsilon_1|^p \sh(s) ds \right) \left( \int_r^{\infty}  (\sh(s))^{-\frac{p^{\prime}}{p}} ds\right)^{\frac{p}{p^{\prime}}}   r dr \\ 
 & \lesssim  \left\| \varepsilon_1 \right\|_{L^p}^p +  \left\| \varepsilon_1 \right\|_{L^p}^p \left\| \varepsilon  \right\|_{L^\infty}^p
 \end{align*} 
 For large $r,$ we obtain
 \begin{align*}
 \int_{\delta}^{\infty}  |A_0|^p \sh(r) dr & \lesssim  \int_{\delta}^{\infty}  \left|  \int_r^{\infty} \re(\varepsilon_1) ds  +    \int_r^{\infty} \re(\varepsilon_1 \bar{\varepsilon}) \, ds    \right|^p e^r dr \\    
  & \lesssim  \int_{\delta}^{\infty} \left| \int_r^{\infty} e^{(r-s)\frac{1}{p }} \re(\varepsilon_1)e^{\frac{1}{p}s}  + \left\|\varepsilon \right\|_{L^{\infty}}\int_r^{\infty}   e^{(r-s)\frac{1}{p }} \re(\varepsilon_1) e^{\frac{1}{p}s} \right|^p dr 
 \\
 & \lesssim  \int_{\delta}^{\infty} \left| \int_{\delta}^{\infty}   \mathbb{1}_{ \{r \leq s \}} e^{(r-s)\frac{1}{p}} \re(\varepsilon_1)e^{\frac{s}{p}} ds \right| dr \left(1+ \left\| \varepsilon \right\|_{L^{\infty}}^p \right)
 \\
 & \lesssim \int_{\delta}^{\infty}  \left| \int_{\delta}^{\infty} K(s,r) \re(\varepsilon_1)e^{\frac{s}{p}} ds  \right|^p   dr  \left(1+ \left\| \varepsilon \right\|_{L^{\infty}}^p \right) \\
 \end{align*}
where 
\begin{align*}
 K(s,r)=   \mathbb{1}_{ \{r \leq s \}} e^{(r-s)\frac{1}{p}}
\end{align*}
Note that $\left\| K(s,r) \right\|_{L^1(ds)} $ bounded uniformly in $r$ and $\left\| K(s,r) \right\|_{L^1(dr)} $ bounded uniformly in $s.$
Then by Schur's test, we have
\begin{align*}
    \int_{\delta}^{\infty}  \left| \int_{\delta}^{\infty} K(s,r) \re(\varepsilon_1)e^{\frac{s}{p}} ds  \right|^p   dr \lesssim   \int_{\delta}^{\infty} 
|\varepsilon_1|^p e^{s}ds \lesssim \left\| \varepsilon_1 \right\|_{L^p}^p
\end{align*}
which yields, 

\begin{align*}
  \int_{\delta}^{\infty}  |A_0|^p \sh(r) dr  \lesssim   \left\| \varepsilon_1 \right\|_{L^p}^p   \left(1+ \left\| \varepsilon \right\|_{L^{\infty}}^p \right)  \lesssim  \left\| \varepsilon_1 \right\|_{L^p}^p +  \left\| \varepsilon_1 \right\|_{L^p}^p \left\| \varepsilon  \right\|_{L^\infty}^p
\end{align*}

\end{proof}

Next we estimate $\|\partial_r\varepsilon\|_{L^p}$ and $\|\frac{\varepsilon}{\sh r}\|_{L^p}$.
\begin{lemma}\label{lem:epH1mtoep1}
    For $2\leq p<\infty$, 
\begin{align} \label{eq:espsi/sh-p}
\left\| \frac{\varepsilon}{\sh(r)} \right\|_{L^p} \lesssim    \left\| \varepsilon_1 \right\|_{L^p}+(1+\|\varepsilon\|_{L^\infty})\left\|\frac{a_\theta}{\sh r}\right\|_{L^p}.
\end{align}
Moreover, 
\begin{align} \label{eq:Epsi-prime-p}
\left\| \partial_r \varepsilon \right\|_{L^p} \lesssim \left\| \varepsilon_1 \right\|_{L^p} +(1+\|\varepsilon\|_{L^\infty})  \left\|  \frac{a_{\theta}}{\sh(r)}\right\|_{L^p}.
\end{align}
\end{lemma}
\begin{proof}
    
Using \eqref{partial_eps_equ}, we have 
\begin{align*}
\frac{\varepsilon}{\sh(r)}= \frac{Q(r)}{\sh(r)} \int_r^{\infty} \frac{\varepsilon_1(s)}{Q(s)} ds - \frac{Q(r)}{\sh(r)} \int_r^{\infty} \frac{a_{\theta}(s)}{\sh(s)} ds - \frac{Q(r)}{\sh(r)} \int_r^{\infty} \frac{1}{Q(s)} \frac{a_{\theta}(s)}{\sh(s)} \varepsilon(s) ds 
\end{align*}
Therefore, 

\begin{align*}
 \int_0^1 \bigg| \frac{\varepsilon}{\sh(r)} \bigg|^p \sh(r) dr & \lesssim \int_0^1 \bigg| \frac{Q(r)}{\sh(r)} \int_r^{\infty} \frac{\sh(s)^{-\frac{1}{p}}}{Q(s)}   \varepsilon_1(s) \sh(s)^{\frac{1}{p}} ds \bigg|^p   \sh(r) dr \\
 & + \int_0^1 \bigg| \frac{Q(r)}{\sh(r)} \int_r^{\infty} \frac{a_{\theta}(s)}{\sh(s)}   \sh(s)^{-\frac{1}{p}} \sh(s)^{\frac{1}{p}} ds \bigg|^p   \sh(r) dr\\
  & + \int_0^1 \bigg| \frac{Q(r)}{\sh(r)} \int_r^{\infty} \frac{1}{Q(s)} \frac{a_{\theta}(s)}{\sh(s)} \varepsilon(s)  \sh(s)^{-\frac{1}{p}} \sh(s)^{\frac{1}{p}} ds \bigg|^p   \sh(r) dr \\
  \lesssim I_1 + I_2 + I_3 
\end{align*}

Let $K_1(s,r):=\mathbb{1}_{\{ r\leq s  \}  } \mathbb{1}_{ \{ 0 \leq r \leq \min(1,s) \}   }     \frac{Q(r)}{\sh(r)} \sh^{\frac{1}{p}}(r)  \frac{\sh(s)^{-\frac{1}{p}}}{Q(s)}    $
Notice that, $\left\| K_1( r , \cdot ) \right\|_{L^1(r,\infty)}$ and $\left\| K_1( \cdot , s ) \right\|_{L^1(0,\min(s,1)}$ are uniformly bounded in $r$ and $s,$ respectively then 
\begin{align*}
|I_1| \lesssim \left\| \varepsilon_1 \right\|_{L^p}^p.
\end{align*}

Let $K_2(s,r):= \mathbb{1}_{\{ r\leq s  \}  } \mathbb{1}_{ \{ 0 \leq r \leq \min(1,s) \}}  \frac{Q(r)}{\sh(r)} \sh^{\frac{1}{p}}(r) \sh(s)^{-\frac{1}{p}} $
Notice that, $\left\| K_2( r , \cdot ) \right\|_{L^1(r,\infty)}$ and $\left\| K_2( \cdot , s ) \right\|_{L^1(0,\min(s,1)}$ are uniformly bounded in $r$ and $s,$ respectively then 
\begin{align*}
|I_2| \lesssim \left\| \frac{a_{\theta}}{\sh(r)} \right\|_{L^p}^p.
\end{align*}

Let $K_3(s,r):= \mathbb{1}_{\{ r\leq s  \}  } \mathbb{1}_{ \{ 0 \leq r \leq \min(1,s) \}}  \frac{Q(r)}{\sh(r)} \sh^{\frac{1}{p}}(r)  \frac{\sh(s)^{-\frac{1}{p}}}{Q(s)} $
Notice that, $\left\| K_3( r , \cdot ) \right\|_{L^1(r,\infty)}$ and $\left\| K_3( \cdot , s ) \right\|_{L^1(0,\min(s,1)}$ are uniformly bounded in $r$ and $s,$ respectively then 
\begin{align*}
|I_3| \lesssim \left\| \frac{a_{\theta}}{\sh(r)} \varepsilon \right\|_{L^p}^p.
\end{align*}

Therefore, we have 
\begin{align*}
\left\| \frac{\varepsilon}{\sh(r)} \right\|_{L^p(0,1)} \lesssim    \left\| \varepsilon_1 \right\|_{L^p}+ \left\| \frac{a_{\theta}}{\sh(r)} \right\|_{L^p} + \left\| \frac{a_{\theta}}{\sh(r)} \right\|_{L^{p}} \left\|  \varepsilon \right\|_{L^\infty} 
\end{align*}
This concludes the proof of \eqref{eq:espsi/sh-p}. Next, recall that by\eqref{partial_eps_equ}, we have    
\begin{align*}
\partial_r \varepsilon = \varepsilon_1 - Q  \frac{a_{\theta}}{\sh(r)} - \frac{a_{\theta}}{\sh(r)} \varepsilon + \varepsilon \frac{\partial_r Q}{Q}
\end{align*}
Thus, 
\begin{align*}
 \left\| \partial_r \varepsilon \right\|_{L^p} \lesssim \left\| \varepsilon_1 \right\|_{L^p}+ \left\| \frac{\varepsilon}{\sh(r)} \right\|_{L^p} +(1+\|\varepsilon\|_{L^\infty})  \left\|  \frac{a_{\theta}}{\sh(r)}\right\|_{L^p}.
\end{align*}
Together with \eqref{eq:espsi/sh-p} this concludes the proof of \eqref{eq:Epsi-prime-p}.
\end{proof}
The next lemma will be used in the proof of both Proposition~\ref{main-prop} and Theorem~\ref{main-theo}.
\begin{lemma}\label{lem:epLptoepH1m}
    We have $\left\| \varepsilon \right\|_{L^2 \cap L^{\infty} }  \lesssim \|\varepsilon\|_{H^1_m}$. Moreover, if $\left\| \varepsilon \right\|_{H^1_m}\leq \delta$, and $\delta$ is sufficiently small, then $\left\| \varepsilon_1\right\|_{L^2}\lesssim \delta$. 
\end{lemma}
  \begin{proof}
Notice that $L^2$-bound of $\varepsilon$ is straightforward by the fact that  $ \left\|\varepsilon \right\|_{L^2} \leq \left\| \partial_r \varepsilon \right\|_{L^2}$ on $\mathbb{H}^2$. For $L^{\infty},$ we can write 
\begin{align*}
  \varepsilon^2(r) &=2 \int_0^r \frac{\varepsilon(s)}{\sh(s)} \sh^{\frac{1}{2}}(s) \partial_s \varepsilon \sh^{\frac{1}{2}}(s) ds  \lesssim \left\| \frac{\varepsilon}{\sh(r)} \right\|_{L^2}
  \left\| \partial_r \varepsilon \right\|_{L^2} .
\end{align*}
Therefore, we obtain $\left\| \varepsilon \right\|_{L^2 \cap L^{\infty}}$ is small. Next, we estimate the $L^2$-norm of $\varepsilon_1,$ recall that $\varepsilon_1=D_{+}(Q+\varepsilon)=\partial_r \varepsilon + \frac{A_{\theta}[Q]-m}{\sh(r)} \varepsilon + \frac{a_{\theta}}{\sh(r)}(Q+\varepsilon), $ Then we have 
\begin{align*}
 \left\| \varepsilon_1 \right\|_{L^2 } \lesssim \left\|  \varepsilon \right\|_{H^1_m} + \left\| \frac{a_{\theta}}{\sh(r)} \right\|_{L^2} \bigg(\left\| Q \right\|_{L^{\infty}} + \left\|  \varepsilon \right\|_{L^{\infty}}    \bigg)  .
\end{align*}
Using Lemma~\ref{claim-estim_a_theta/sinh}, we obtain that $\left\| \varepsilon_1 \right\|_{L^2}\lesssim\delta$.
  \end{proof}
We are now ready to prove Proposition~\ref{main-prop}.

\begin{proof}[Proof of Proposition~\ref{main-prop} ]
    For simplicity of notation let 
    \begin{align*}
        X:=\sup_{t\in I}\|\varepsilon \|_{H^1_m}+\|\varepsilon_1\|_{\mathcal{S}(I)}.
    \end{align*}
    We will prove that for some $b>1$
    \begin{align}\label{eq:XStrichartzgoal1}
        X\lesssim \|\varepsilon(0)\|_{H^1_m}+\|\varepsilon_1(0)\|_{L^2}+X^b.
    \end{align}
    If $\delta$ is sufficiently small, this proves the desired result. Using \eqref{eq-2-epsilon1}, Corollary~\ref{Strichartz},
\begin{align*}
 \left\| \varepsilon_1 \right\|_{S} &\lesssim  \left\| \varepsilon_1(0) \right\|_{L^2} + \left\| N(\varepsilon_1) \right\|_{L^{\frac{4}{3}} L^{\frac{4}{3}}}  \\
 &\lesssim \left\| \varepsilon_1(0) \right\|_{L^2} + \left\| A_0  \right\|_{L^{\frac{8}{3}} L^{\frac{8}{3}}}
 \left\| \varepsilon_1  \right\|_{L^{\frac{8}{3}} L^{\frac{8}{3}}} +   \left\|\varepsilon  \right\|_{L^{\frac{8}{3}} L^{\frac{8}{3}}} \left\|\varepsilon_1   \right\|_{L^{\frac{8}{3}} L^{\frac{8}{3}}} +\left\|\varepsilon_1  \right\|_{L^{4} L^{4}} \left\|\varepsilon  \right\|^2_{L^{4} L^{4}} \\
 & \quad+ \left\|  \coth(r)  \frac{a_{\theta}}{\sh(r)}     \varepsilon_1 \right\|_{L^{\frac{4}{3}} L^{\frac{4}{3}}}\Big(1+\Big\|\frac{A_\theta[Q]-m}{\ch (r)
}\Big\|_{L^\infty}\Big)  + \left\| \frac{a_{\theta}^2}{\sh^2(r)} \varepsilon_1  \right\|_{L^{\frac{4}{3}} L^{\frac{4}{3}}}.
\end{align*}
We can combine Lemma \ref{lem:epsi-LP-epsi1L_p}, \ref{claim-estim_a_theta/sinh}, \ref{lem:est-coth_a_theta/sinh-eps_1}, \ref{lem:A_0-Lp} to estimate the right-hand side above. 
Indeed, the most delicate terms are the ones on the last line above which can be bounded as follows. Observe that by the first part of Lemma~\ref{lem:epLptoepH1m}, we have $\|\varepsilon\|_{L^p}\lesssim \|\varepsilon\|_{L^\infty}+\|\varepsilon\|_{L^2}\lesssim \|\varepsilon\|_{H^1_m}$ for all $2\leq p \leq \infty$. Then,
\begin{align*}
    \left\|\frac{a_\theta}{\sh r}\varepsilon_1\right\|_{L^{\frac{4}{3}}L^{\frac{4}{3}}}\leq \left\|\frac{a_\theta}{\sh r}\right\|_{L^\infty L^4}^2\|\varepsilon\|_{L^{\frac{4}{3}} L^4}\lesssim X^3+X^5.
\end{align*}
Also
\begin{align*}
    \left\|  \coth(r)  \frac{a_{\theta}}{\sh(r)}     \varepsilon_1 \right\|_{L^p} \lesssim(1+\|\varepsilon\|_{L^\infty L^\infty })(\left\| \varepsilon \right\|_{L^\infty L^\infty}  + \left\| \varepsilon \right\|_{L^\infty L^{2}}) \left\|  \varepsilon_1  \right\|_{L^{\frac{4}{3}} L^4 }\lesssim X^2+X^3.
\end{align*}
This bounds $\|\varepsilon_1\|_{\mathcal{S}(I)}$ in terms of the initial data and higher powers of $X$. To estimate $\sup_{t\in I}\|\varepsilon\|_{H^1_m}$ we apply Lemma~\ref{lem:epH1mtoep1} with $p=2,$ together with Lemma\ref{lem:epsi-LP-epsi1L_p} and Lemma~\ref{claim-estim_a_theta/sinh}, to obtain
\begin{align*}
    \sup_{t\in I}\|\varepsilon\|_{H^1_m}\lesssim \sup_{t\in I}\|\varepsilon_1\|_{L^2}+X^2+X^3\leq \|\varepsilon_1\|_{\mathcal{S}(I)}+X^2+X^3.
\end{align*}
This concludes the proof of the desired estimate \eqref{eq:XStrichartzgoal1}.
\end{proof}

\subsection{Proof of Theorem~\ref{main-theo}}
\label{subsec:Proof-of-Theo}
Finally, we are in a position to prove Theorem~\ref{main-theo}.
\begin{proof}[Proof of Theorem~\ref{main-theo} ]
    Note that $\|\varepsilon\|_{L^\infty\cap L^2}\lesssim \|\varepsilon\|_{H^1_m}$ by the first part of Lemma~\ref{lem:epLptoepH1m}. Moreover, by the second part of the same lemma, if $\delta$ is sufficiently small, then $\|\varepsilon_1(0)\|_{L^2}\lesssim \delta$. It follows that the hypotheses of Proposition~\ref{main-prop} are satisfied and hence
    \begin{align}\label{eq:Strichinthm1}
        \sup_{t\in I} \|\varepsilon\|_{H^1_m}+\|\varepsilon_1\|_{\mathcal{S}(I)}\lesssim \delta,
    \end{align}
    with a constant independent of $\delta$, on any interval on which the solution exists. Therefore the solution exists globally and, we can take $I=\mathbb{R}$ in \eqref{eq:Strichinthm1}. It remains to prove that $\|\varepsilon\|_{\mathcal{S}^1_m(\mathbb{R})}\lesssim \delta$. But this is now a  direct consequence of \eqref{eq:Strichinthm1} and Lemma~\ref{lem:epH1mtoep1}.
\end{proof}
\appendix
\section{Well-posedness}\label{app:LWP}
In this section, we prove that the equation \eqref{eq:vareplinintro1} for $\varepsilon(r)$ is locally well posed with an arbitrary data in $H^1_m,$ via a fixed point argument. We rely on the Strichartz estimates for the free Schr\"odinger equation established in \cite{AnkerJeanPierfelice09}. For this purpose, we rewrite \eqref{eq:vareplinintro1} as follows:
We decompose the $m$-equivariant solution $(\Phi,A_{\theta})$ to  \eqref{GL-g-2}-\eqref{eq_B_2} as in \eqref{eq:decom1}, i.e., Let $\Phi=e^{ im\theta} \phi,$ where $\phi=Q+\varepsilon$ and  $A_{\theta}=A_{\theta}[Q] + a_{\theta}.$ 
Plugging this decomposition into the equation \eqref{GL-g-2} and using \eqref{eq_B_2} with \eqref{eq:Q}, we obtain the corresponding linearized equation for $\varepsilon.$ 

\begin{align}
\label{eq:varepsilon}
 \big( i  \partial_t + \frac{1}{2 } \mathbf{\Delta} \big) e^{i m \theta}  \varepsilon(t,r) = e^{i m \theta}  \mathcal{F}(\varepsilon(t,r)),
\end{align}
where $\mathbf{\Delta}= -\partial_r^{\ast} \partial_r + \frac{1}{\sh^2(r)} \partial_{\theta}^2$ and 
\begin{align*}
 \mathcal{F}(\varepsilon)   &= - \frac{1}{4}(1-Q^2) \varepsilon + \frac{1}{2} (Q\re(\varepsilon)+\frac{1}{2}|\varepsilon|^2) (Q+\varepsilon) + \frac{a_{\theta}}{\sh(r)} \frac{A_{\theta}[Q]-m}{\sh(r)} (Q+ \varepsilon) \\
&  + \frac{1}{2} (\frac{A_{\theta}[Q]^2-2 m A_{\theta}[Q]}{\sh^2(r)}) \varepsilon  + \frac{1}{2}(\frac{a_{\theta}}{\sh(r)})^2 (Q+ \varepsilon) - A_0 (Q+\varepsilon).
\end{align*}

\begin{theorem} 
\label{theo:well-posed}
Let $\Phi_0=e^{im \theta} \phi_0(r) \in H^1$ such that $\phi_0=Q+\varepsilon_0.$ Then there exists $\tau >0$ depending on $\left\| \varepsilon_0 \right\|_{H^1_m}$ such that for all $T \in (0,\tau)$ the equation \eqref{GL-g-2} has a unique solution of the form $\Phi=e^{i m \theta}(Q+\varepsilon(t)),$ where $\varepsilon(t)$ is the unique solution of \eqref{eq:varepsilon} in $\mathcal{C}((0,T),H^1_m(\HH^2)) \cap \mathcal{S}^1_m((0,T),\HH^2).$
\end{theorem}

\begin{proof}
Notice that the Strichartz estimates from \cite[Theorem~3.6]{AnkerJeanPierfelice09} can be combined into the following estimate
\begin{align*}
 \left\|\varepsilon \right\|_{\mathcal{S}^1_m} \lesssim   \left\| \varepsilon_0 \right\|_{H^1_m}  + \left\|  \mathcal{F}(\varepsilon) \right\|_{\mathcal{N }^1_m}, 
\end{align*}

Fix $M>0$ to be specified later.  Let $B$ be the ball of $X=C([0,T],H^1_m) \cap \mathcal{S}^1_m((0,T), \HH^2)$, with radius $M>0$ and  center $0$, i.e the set of functions $u \in X$ such that 

 $$\left\|u \right\|_{L^{\infty}H^1_m} \leq M   \text{ and }   \left\|u \right\|_{\mathcal{S}^1_m } \leq M .$$ 
Denote 
$$ d_{B}(u,v)= \left\|u-v  \right\|_{L^{\infty}H^1_m} + \left\| u-v\right\|_{\mathcal{S}^1_m} $$

Then one can check that $(B,d_B)$ is a complete metric space. 
For $\varepsilon \in B $ we define $\Phi(\varepsilon)(t):= e^{it\Delta } e^{im \theta } \varepsilon_0 + D(\varepsilon)(t) ,$ where $D(\varepsilon)$ is the Duhamel term given~by
 $$  D(v)(t):= - \,i \,\displaystyle \int_{0}^{t} e^{i(t-s)\Delta} e^{im \theta } \mathcal{F}(\varepsilon(s)) ds  .  $$ 
\begin{itemize}
\item Step 1 : Stability of $B$.  \\
We will prove that: for $\varepsilon \in B \Longrightarrow \Phi(\varepsilon) \in B ,$ for a good choice of $M$ and $T$. \\
We have 

$$\left\|e^{it \Delta } \varepsilon_0 \right\|_{L^{\infty}H^1_m }= \left\| \varepsilon_0 \right\|_{H^1_m}   \leq \frac{M}{4}  .$$
If the following conditions satisfied \begin{equation}
\label{estimate_h^s}
   4 \left\|\varepsilon_0 \right\|_{ H^1_m} \leq M ,
\end{equation}
and using Strichartz estimate, we obtain 

\begin{equation*}
\left\|e^{it\Delta}\varepsilon_0 \right\|_{\mathcal{S}^1_m} \leq C_s  \left\|\varepsilon_0 \right\|_{H^1_m} \leq \frac{M}{4} .
\end{equation*} 
If $M$ is chosen so that \begin{equation}
\label{estimate_L^a_w^s,p}
    M \geq 4 C_s \left\|u_0 \right\|_{H^1_m}.
\end{equation}
Hence, we have  
\begin{equation}
\label{max-M/2}
    \max \left( \left\|e^{it \Delta } \varepsilon_0 \right\|_{L^{\infty}H^1_m },\left\|e^{it\Delta}\varepsilon_0 \right\|_{\mathcal{S}_m^1 }  \right) \leq \frac{M}{4}, \newline \; 
\end{equation}
Provided \eqref{estimate_h^s},  \eqref{estimate_L^a_w^s,p} are satisfied.

Next, we treat the nonlinear term. We first claim the following estimate, which will be used throughout the proof, and we postpone its proof to the end of this section. 
  \begin{claim}  \label{claim:a_theta-est}
Let $\alpha=1,2$, we have for any $2\leq p \leq \infty$
\begin{align*}
   \left\| \frac{a_\theta}{\sh^{\alpha}(r)} \right\|_{L^p} \lesssim  \left\| \varepsilon \right\|_{L^{2} }^{2}  + \left\| \varepsilon \right\|_{L^2 } + \left\| \frac{\varepsilon}{\sh(r)} \right\|_{L^2}^2  
\end{align*}
\end{claim}

We will first estimate the nonlinear term $\mathcal{F}$ in $L^2$-norm. 
 \begin{claim} \label{claim:F_S}
Let $0<T<1,$ then 
  \begin{align*}
\left\| \frac{1}{\sh(r)} \mathcal{F} (\varepsilon) \right\|_{\mathcal{N}} \leq C_1 T^{\frac{1}{2}} (\left\| \varepsilon \right\|_{\mathcal{S}^1_m}+\left\| \varepsilon \right\|_{\mathcal{S}^1_m}^5)
\end{align*} 
\end{claim}
\begin{proof}
Using the fact that $\frac{1}{\sh(r)}\partial_r A_{\theta}[Q]= \frac{1}{2}(1-Q^2),$ and $\frac{1}{\sh(r)}\partial_r a_{\theta}=-Q\re(\varepsilon)-\frac{1}{2}|\varepsilon|^2$,
   \begin{align*} 
   \left\| \frac{1}{4}(1-Q^2) \frac{\varepsilon}{\sh(r)} + Q^2 \re(\frac{\varepsilon}{\sh(r)}) \right\|_{L^1 L^2 } &\lesssim \bigg( \left\| 1-Q^2 \right\|_{L^1 L^{\infty}}  + \left\| Q^2 \right\|_{L^1 L^{\infty}} \bigg) \left\|  \frac{\varepsilon}{\sh(r)} \right\|_{L^{\infty}L^2} \\
   & \lesssim  T \left\|  \frac{\varepsilon}{\sh(r)} \right\|_{L^{\infty}L^2}  
   \end{align*} 

Using the fact that $Q(r)=O(r^{m}),$ as $r \to 0$ we have
\begin{align*}
 \left\|   \frac{Q}{\sh(r)}|\varepsilon|^2 \right\|_{L^{\frac{4}{3}} L^{\frac{4}{3}} } &\lesssim \left\| \frac{Q}{\sh(r)} \right\|_{L^2 L^{\infty}} \left\| |\varepsilon|^2 \right\|_{L^{4}L^{\frac{4}{3}}} \lesssim T^{\frac{1}{2}}  \left\| \varepsilon \right\|_{L^{8}L^{\frac{8}{3}}}^2
 \end{align*}

\begin{align*}
 \left\|   \frac{\varepsilon}{\sh(r)}  |\varepsilon|^2 \right\|_{L^{\frac{4}{3}} L^{\frac{4}{3}} } &\lesssim   \left\|  \frac{\varepsilon}{\sh(r)}   \right\|_{L^{4}L^{4}}   \left\| \varepsilon \right\|_{L^{4}L^{4}}^2 
 \end{align*}

Using the fact that $A_{\theta}[Q]=\int_0^r \frac{1}{2}(1-Q^2) \sh(s) ds $ and $Q(r)=O(r^{m}),$ as $r \to 0,$ we obtain 
 \begin{align*}
 \left\|   \frac{A_{\theta}[Q]^2 -2 m A_{\theta}[Q] }{\sh^2(r)}   \frac{\varepsilon}{\sh(r)}   \right\|_{L^{1} L^{2} } &\lesssim  \left\|   \frac{A_{\theta}[Q]^2 -2 m A_{\theta}[Q] }{\sh^2(r)} \right\|_{L^{1}L^{\infty}} 
 \left\|   \frac{\varepsilon}{\sh(r)}   \right\|_{L^{\infty}L^{2}}  \\
 & \lesssim T  \left\|   \frac{\varepsilon}{\sh(r)}   \right\|_{L^{\infty}L^{2}}
 \end{align*}

Using Claim \ref{claim:a_theta-est}, we obtain 
\begin{align*}
 \left\|   \frac{A_{\theta}[Q]-m }{\sh(r)} \frac{a_{\theta}}{\sh(r) }  \frac{\varepsilon}{\sh(r)}  \right\|_{L^{\frac{4}{3}} L^{\frac{4}{3}} } &\lesssim \left\|     \frac{a_{\theta}}{\sh^2(r) }   \right\|_{L^2L^{2}} 
 \left\|  \frac{\varepsilon}{\sh(r)}  \right\|_{L^{4}L^{4}}  \\
 &\lesssim T^{\frac{1}{2}} \left\|  \frac{\varepsilon}{\sh(r)}  \right\|_{L^{4}L^{4}}  \bigg(  
      \left\| \varepsilon \right\|_{L^{\infty} L^{2} }^{2} +   \left\| \varepsilon \right\|_{L^{\infty} L^{2} }  +  \left\| \frac{\varepsilon}{\sh(r)} \right\|_{L^{\infty} L^{2}}^2   \bigg)
 \end{align*}
Using the asymptotics of $Q$ and Claim \ref{claim:a_theta-est}, we get 
\begin{align*}
 \left\|   \frac{A_{\theta}[Q]-m }{\sh(r)} \frac{a_{\theta}}{\sh(r) }  \frac{Q}{\sh(r)}  \right\|_{L^{\frac{4}{3}} L^{\frac{4}{3}} } &\lesssim \left\|     \frac{a_{\theta}}{\sh^2(r) }   \right\|_{L^2L^{2}} 
 \left\|  \frac{Q}{\sh(r)}  \right\|_{L^{4}L^{4}}  \\
 &\lesssim T^{\frac{1}{4}}  T^{\frac{1}{2}}  \bigg(  
    \left\| \varepsilon \right\|_{L^{\infty} L^{2} }^{2} +   \left\| \varepsilon \right\|_{L^{\infty} L^{2} }  +  \left\| \frac{\varepsilon}{\sh(r)} \right\|_{L^\infty L^2}^2   \bigg) 
 \end{align*}

Similarly, we have 

\begin{align*}
  \left\| \left(\frac{a_{\theta}}{\sh(r)} \right)^2 \frac{Q}{\sh(r)} \right\|_{L^1 L^2} & \lesssim 
   \left\| \frac{a_{\theta}}{\sh(r)} \right\|_{L^{2} L^4}^2 \left\| \frac{Q}{\sh(r)} \right\|_{L^{\infty}} \\
   & \lesssim \left\| \varepsilon \right\|_{L^{4} L^{2} }^{4} +  \left\| \varepsilon \right\|_{L^2 L^{2} }^2 +  \left\| \frac{\varepsilon}{\sh(r)} \right\|_{L^{4} L^{2}}^4 \\
& \lesssim T   \bigg( \left\| \varepsilon \right\|_{L^\infty L^{2} }^4 +    \left\| \varepsilon \right\|_{L^{\infty} L^{2} }^{2}  + \left\| \frac{\varepsilon}{\sh(r)} \right\|_{L^\infty L^{2}}^4 \bigg)
\end{align*}

\begin{align*}
  \left\| \left(\frac{a_{\theta}}{\sh(r)} \right)^2 \frac{\varepsilon}{\sh(r)} \right\|_{L^1 L^2 }& \lesssim 
   \left\| \frac{a_{\theta}}{\sh(r)} \right\|_{L^{2} L^{\infty} }^2 \left\| \frac{\varepsilon}{\sh(r)}  \right\|_{L^{\infty}L^2} \\
   & \lesssim  \bigg( \left\| \varepsilon \right\|_{L^{4} L^{2} }^{4} +  \left\| \varepsilon \right\|_{L^2 L^{2} }^2 +  \left\| \frac{\varepsilon}{\sh(r)} \right\|_{L^{4} L^{2}}^4 \bigg) \left\| \frac{\varepsilon}{\sh(r)}  \right\|_{L^{\infty}L^2} \\
   & \lesssim T  \left\| \frac{\varepsilon}{\sh(r)}  \right\|_{L^{\infty}L^2} \bigg( \left\| \varepsilon \right\|_{L^\infty L^{2} }^4 +    \left\| \varepsilon \right\|_{L^{\infty} L^{2} }^{2}  + \left\| \frac{\varepsilon}{\sh(r)} \right\|_{L^\infty L^{2}}^4 \bigg)   \\ 
\end{align*}

Finally, we estimate the last nonlinear term $A_0\varepsilon.$ First, notice that by \eqref{eq_E_r_2} and \eqref{eq_B_2}, we have $\partial_r A_0=-\frac{1}{2} \re ( \bar{\phi} \title{D}_{+} \phi) .$ Integrating from $r$ to $\infty$ and using the decomposition of $\phi$ we obtain
\begin{align*}
  A_0 &= \frac{1}{2 } \int_r^{\infty}   \re\bigg( (  \partial_r \varepsilon + \frac{A_{\theta}[Q]-m}{\sh(s)} \varepsilon + \frac{a_{\theta}}{\sh(s)} )   ( Q + \bar{\varepsilon}) \bigg)ds.
\end{align*}
\begin{claim}
\label{clm:A_0}
    Let $\frac{4}{3}<p<4, $ and $q>2$ then we have 
    \begin{align*}
     \left\| \int_r^{\infty}   \re \bigg(  (\partial_r \varepsilon + \frac{A_{\theta}[Q]-m}{\sh(s)} \varepsilon +\frac{a_{\theta}}{\sh(s)}  )   \bar{\varepsilon} \bigg)  ds  \right\|_{L^p}  &  \lesssim \left\| \varepsilon \right\|_{H^1_m}^2  + \left\| \frac{a_{\theta}}{\sh(r)} \right\|_{L^4} \left\|\varepsilon \right\|_{L^2} \\
     \left\| \int_r^{\infty}   \re \bigg(  (\partial_r \varepsilon + \frac{A_{\theta}[Q]-m}{\sh(s)} \varepsilon +\frac{a_{\theta}}{\sh(s)}   )  Q \bigg)  ds  \right\|_{L^q} & \lesssim \left\| \varepsilon \right\|_{H^1_m}  +  \left\| \frac{a_{\theta}}{\sh(r)} \right\|_{L^2}
     \end{align*}

\end{claim}
\begin{proof}
  \begin{align*}
\int_0^{\infty} \left| \int_r^{\infty}   \re \bigg(  (\partial_r \varepsilon + \frac{A_{\theta}[Q]-m}{\sh(s)} \varepsilon + \frac{a_{\theta}}{\sh(s)}  )   \bar{\varepsilon} \bigg)  ds \right|^p \sh(r) dr  & \lesssim \bigg( \left\| \partial_r \varepsilon + \frac{A_{\theta}[Q]-m}{\sh(r)} \varepsilon \right\|_{L^2}^p \left\| \varepsilon \right\|_{L^4}^p \\
&\qquad \quad   + \left\| \frac{a_{\theta}}{\sh(r)} \right\|_{L^4}^p \left\|\varepsilon \right\|_{L^2}^p \bigg) \\
& \times \int_0^{\infty}  
\left(  \int_r^{\infty} \sh(s)^{-3}  ds \right)^{\frac{p}{4}} \sh(r)dr 
\end{align*}

 \begin{align*}
\int_0^{\infty} \left| \int_r^{\infty}   \re \bigg(  (\partial_r \varepsilon + \frac{A_{\theta}[Q]-m}{\sh(s)} \varepsilon +\frac{a_{\theta}}{\sh(s)}  )  Q \bigg)  ds \right|^q \sh(r) dr & \lesssim \bigg(  \left\| \partial_r \varepsilon + \frac{A_{\theta}[Q]-m}{\sh(r)} \varepsilon \right\|_{L^2}^q  \\
& \qquad \quad  + \left\| \frac{a_{\theta}}{\sh(r)} \right\|_{L^2}^q  \bigg) \\
& \times \int_0^{\infty}  
\left(  \int_r^{\infty} \sh(s)^{-1} Q^2(s)  ds \right)^{\frac{q}{2}} \sh(r)dr  
\end{align*} 
\end{proof}

Using Claim \ref{clm:A_0} and \ref{claim:a_theta-est} and similar estimate as for other nonlinear terms we obtain the desired result. This completes the proof of Claim \ref{claim:F_S},

\end{proof}

 \begin{claim} \label{clm:F'_S}
 Let $0<T<1,$ then 
\begin{align*} 
\left\| \partial_r \mathcal{F} (\varepsilon) \right\|_{\mathcal{N }} \leq C_2 T^{\frac{1}{2}} \left( \left\| \varepsilon \right\|_{\mathcal{S}_m^1} +  \left\| \varepsilon \right\|_{\mathcal{S}_m^1}^5\right)
\end{align*} 
\end{claim}
\begin{proof}
This is very similar to the analogous estimate for $\frac{F}{\sh r}$. If the derivative falls on $\varepsilon$, $A_0$, or $a_\theta$, then we estimate the corresponding terms as before, except that the resulting coefficients are less singular at $r=0$. Note that we did not use the decay of $(\sh r)^{-1}$ for large $r$ in estimating $\frac{F}{\sh r}$ above. When the derivative falls on the coefficients, then at worst this produces an extra $(\sh r)^{-1}$ singularity near $r=0$ which was already treated above. 
\end{proof}

Using Claim \ref{claim:F_S} and \ref{clm:F'_S}, we obtain 
\begin{align*}
\left\| \mathcal{F}(\varepsilon) \right\|_{\mathcal{N }^1} \leq \frac{M}{4},
\end{align*}
Provided, \begin{equation} \label{eq:M-cond-3}
 C \, T^{\frac{1}{2}} M^4 \leq \frac{1}{4}, \quad \text{where} \; C=\max(C_1,C_2)
\end{equation}

\item Step 2: Contraction property. \\
Using similar estimate as in step 1, one can check that for $0<T<1$ and for $u,v \in B$, we have 
  \begin{align*}
 \left\| \mathcal{F}(u) - \mathcal{F}(v) \right\|_{\mathcal{N }^1} \leq \frac{1}{2} \left\| u- v \right\|_{\mathcal{S}_m^1}    , 
  \end{align*}
Provided 
\begin{align}
    \label{eq:cond-contra} 
    0<C T^{\frac{1}{2}} M^{\alpha} <1, \quad \text{for some } \,\alpha>1.
\end{align}
Therefore, $\Phi$ is a contraction on the metric space $(B,d_B),$ and   one can choose $M$ and $T<\tau$ such that the conditions \eqref{max-M/2}, \eqref{eq:M-cond-3}, and \eqref{eq:cond-contra} hold. Then by the fixed point theorem there exist a unique solution $\varepsilon(t)$ for the equation \eqref{eq:varepsilon} and therefore a unique solution to \eqref{GL-g-2} of the form $\Phi=e^{i m \theta}(Q+\varepsilon(t)).$ 
\end{itemize}
This completes the proof of Theorem \ref{theo:well-posed}.
\end{proof}

We finish this section with the proof of Claim \ref{claim:a_theta-est}.
\begin{proof}[Proof Claim \ref{claim:a_theta-est} ]
We first estimate the $L^\infty$-norm. Near $r=0$, we consider the more difficult case $\alpha=2$ and for large $r$ the more difficult case $\alpha=1$.
For small $r,$ we have 
\begin{align*}
 \left| \frac{a_{\theta}}{\sh^{2}(r)}   \right| &\lesssim   \frac{1}{\sh(r)^2} \left| \int_0^r  (|\varepsilon|^2 + Q \re(\varepsilon ) ) \sh(s) ds  \right|   \\
 & \lesssim \left|  \int_0^r \frac{|\varepsilon|^2}{\sh^2(s) } \sh(s) ds  \right| + \frac{1}{\sh^2(r)}  \left\| \varepsilon  \right\|_{L^2}   \left(  \int_0^r Q^2 \sh(s) ds  \right)^{\frac{1}{2}} \\
  & \lesssim    \left\| \frac{\varepsilon}{\sh(r)}  \right\|_{L^2}^2 + \left\| \varepsilon  \right\|_{L^2} \frac{1}{r^2} \left( \int_0^r s^{2 m+1} ds \right)^{\frac{1}{2}}
\end{align*}
For large $r,$ we have 
\begin{align*}
 \left| \frac{a_{\theta}}{\sh(r)}   \right| &\lesssim   \frac{1}{\sh(r)} \left| \int_0^r  (|\varepsilon|^2 + Q \re(\varepsilon ) ) \sh(s) ds  \right|  dr \\
 &  \lesssim  \left\| \varepsilon \right\|_{L^2}^2 + \frac{1}{\sh(r)}  \left\| \varepsilon \right\|_{L^2} \bigg( \int_0^r \sh(s) ds   \bigg)^{\frac{1}{2}} \\
 &  \lesssim  \left\| \varepsilon \right\|_{L^2}^2 +   \left\| \varepsilon \right\|_{L^2} .
\end{align*}
Next, we estimate the $L^p$-norms. It suffices to consider large $r$, since the small $r$ regime is controlled by the $L^\infty$ bound established above. We focus on $\alpha=1$, since the case $\alpha=2$ follows similarly as it exhibits stronger decay.

     \begin{align*}
 \int_{1}^{\infty}       \left| \frac{a_\theta}{\sh(r)} \right|^p \sh(r) dr & \lesssim  \int_{1}^{\infty}  \frac{1}{\sh^{p}(r)}     \left|  \int_0^1  \left( |\varepsilon|^2 + Q \re(\varepsilon) \right)  \sh(s) ds  \right|^p \sh(r) dr \\
 & +  \int_{1}^{\infty}  \frac{1}{\sh^{ p}(r)} \left| \int_1^{r} \left( |\varepsilon|^2 + Q \re(\varepsilon) \right)  \sh(s) ds  \right|^p \sh(r) dr \\
 & \lesssim \int_{1}^{\infty} e^{- p r+r}     \left|  \int_0^1   |\varepsilon|^2 \sh(s) ds  + \int_0^1 s^{m+1} \re(\varepsilon)   ds  \right|^p  dr \\
 &+ \int_{1}^{\infty} e^{- p r+r}   \left|  \int_1^{r}   |\varepsilon|^2 \sh(s) ds \right|^p dr + \int_{1}^{\infty}   \left| \int_1^r  \re(\varepsilon) e^s  e^{- r+\frac{r}{p}} ds  \right|^p  dr \\
 & \lesssim   \left\| \varepsilon \right\|_{L^{2} }^{2p} + \left\| \varepsilon \right\|_{L^2 }^p,
\end{align*}
where we have used Young's inequality. Indeed
\begin{align*}
  \int_{1}^{\infty}   \left| \int_1^r  \re(\varepsilon) e^\frac{s}{2}  e^\frac{s}{2} e^{- r+\frac{r}{p}} ds  \right|^p  dr =   \int_{1}^{\infty} \left|  \int_1^\infty K(s,r) \re(\varepsilon(s))e^\frac{s}{2}  ds \right|^p dr   
\end{align*}
where $K(s,r):= \mathbb{1}_{ \{1\leq  s \leq r  \} } e^\frac{s}{2} e^{-\frac{p-1}{p}r }.$ Notice that 
\begin{align*}
 \sup_{r\geq 1} e^{-2 \frac{p-1}{p+2}r} \int_1^r e^{\frac{p}{p+2}s} ds  \leq     \sup_{r\geq 1}e^{-2 \frac{p-1}{p+2}r}  e^{\frac{p}{p+2}r} < \infty,\qquad \sup_{s\geq 1}e^{\frac{sp}{2+p}}\int_s^\infty e^{-2 \frac{p-1}{p+2}r} dr <\infty.
\end{align*}
Then $\left\| K(\cdot,r) \right\|_{L^{\frac{2p}{2+p}} (ds)}$ and $\left\| K(s,\cdot) \right\|_{L^{\frac{2p}{2+p}} (dr)}$ are uniformly bounded in $r$ and $s,$ respectively. Therefore by Young's inequality, we obtain the desired result. 
\end{proof}

\bibliographystyle{acm}
\bibliography{references.bib}

\end{document}